\documentclass[10pt,notitlepage,twoside,a4paper]{amsart}
 \usepackage{amsfonts}

\usepackage{amsmath,amssymb,enumerate}

\usepackage{epsfig,fancyhdr,color}

\usepackage{amssymb}
\usepackage{amsmath,amsthm}
\usepackage{latexsym}
\usepackage{amscd}
\usepackage{psfrag}
\usepackage{graphicx}
\usepackage[latin1]{inputenc}
\usepackage[all]{xy}
\usepackage[mathcal]{eucal}

\definecolor{NoteColor}{rgb}{1,0,0}

\renewcommand{\textsc}{\textcolor{red}}

\newtheorem{theorem}{\rm\bf Theorem}[section]
\newtheorem{proposition}[theorem]{\rm\bf Proposition}
\newtheorem{lemma}[theorem]{\rm\bf Lemma}
\newtheorem{corollary}[theorem]{\rm\bf Corollary}
\newtheorem*{theorem 1}{\rm\bf Proposition 1}
\newtheorem*{theorem 2}{\rm\bf Proposition 2}

\theoremstyle{definition}
\newtheorem{definition}[theorem]{\rm\bf Definition}
\newtheorem{problems}[theorem]{\rm\bf Problems}
\newtheorem{problem}[theorem]{\rm\bf Problem}

\theoremstyle{remark}

\newtheorem{question}[theorem]{\rm\bf Question}

\def\interieur#1{\mathord{\mathop{\kern 0pt #1}\limits^\circ}}

\title[Actions of mapping class groups]{Actions of mapping class groups}

\date{\today}

\author{Athanase Papadopoulos}
\address{Athanase Papadopoulos, Institut de Recherche Math\'ematique Avanc\'ee, Universit{\'e} de Strasbourg and CNRS,
7 rue Ren\'e Descartes,
 67084 Strasbourg Cedex, France;  
Erwin Schr\"odinger International Institute for Mathematical Physics, 
Boltzmanngasse 9, 1090, Wien, Austria}
\email{papadop@math.unistra.fr}

\begin{document}

\maketitle


\thispagestyle{empty}


\begin{abstract}
\vskip 3mm\footnotesize{

\vskip 4.5mm
\noindent
This paper has three parts. The first part is a general introduction to rigidity and to rigid actions of mapping class group actions on various spaces. In the second part, we describe in detail four rigidity results that concern actions of mapping class groups on spaces of foliations and of laminations, namely, Thurston's sphere of projective foliations equipped with its projective piecewise-linear structure, the space of unmeasured foliations equipped with the quotient topology, the reduced Bers boundary, and the space of geodesic laminations equipped with the Thurston topology. In the third part, we present some perspectives and open problems on other actions of mapping class groups.

\vspace*{2mm}

\vspace*{2mm}
}

\vspace*{2mm}
\noindent{This paper will appear in the \emph{Handbook of Group actions}, vol. I, ed. L. Ji, A. Papadopoulos and S.-T. Yau, Higher Eucation Press and International Press, Beijing-Boston, 2014.}

\vspace*{2mm}
\noindent{The author acknowledges support from the French program ANR FINSLER and from the Erwin Schr\"odinger International Institute for Mathematical Phy\-sics (Vienna).}

\end{abstract}

 \section{introduction}

  The mapping class group of an oriented surface $S$ of finite topological type is the group of isotopy classes of orientation-preserving homeomorphisms of $S$, with the group operation inherited from composition of homeomorphisms. The extended mapping class group of $S$ is the group of isotopy classes of all  homeomorphisms of $S$.  The mapping class group is a subgroup of index two of the extended mapping class group.

The study of mapping class groups of surfaces started with the works of Dehn and Nielsen around 1913 and it went on interrupted. Besides being basic objects in surface topology, these groups appear now in several branches of mathematics: combinatorial group theory, 3-manifolds (e.g. through Heegaard splittings), 4-manifolds (e.g. through Lefschetz fibrations), complex geometry (throughout their actions on Teichm\"uller space and other complex spaces, and as monodromies of singularities) and algebraic geometry (works of Weil, Grothendieck, Mumford, etc.). Mapping class groups also appear in theoretical physics (quantum field theories and string theory). These groups admit actions on various spaces which are very different in nature:   combinatorial, algebraic, holomorphic, real-analytic, isometric, piecewise-linear, etc. and in each case the action gives new information on the groups and on the space on which it acts. There is also a profinite theory and an approach through categories and functors (the so-called Grothendieck-Teichm\"uller theory). The diversity of the points of view and the naturality of the actions make the study of mapping class groups an extremely rich subject.

 There are several strong analogies and relations between mapping class groups and other groups, e.g. arithmetic groups (like the groups $\mathrm{SL}(n,\mathbb{Z})$), braid groups, automorphisms of free groups,  CAT(0) groups, hyperbolic groups, Coxeter groups (as a consequence of the fact that a mapping class groups is generated by a finite number of involutions) and Artin groups (which in some sense are generalizations of Coxeter groups).

One famous example of parallel developments of mapping class groups and arithmetic groups is the work done on the stability of the homology or cohomology of these groups and the computation of their stable homology or cohomology. We recall some of the results. Borel proved in \cite{Borel1975}  that the cohomology of arithmetic groups like $\mathrm{SL}(n,\mathbb{Z})$ and their finite index subgroups is stable. Using the combinatorial action of the mapping class group on the curve and the arc complex, Harer proved that the homology groups of the mapping class group are stable \cite{Harer1985}. While Borel's result says that the cohomology of  $\mathrm{SL}(n,\mathbb{Z})$ is independent of $n$, for $n$ large enough, Harer's result says that the homology of a surface mapping class group is independent of the genus $g$ and the number $n$ of boundary components, for $g$  and $n$ large enough. In the same paper \cite{Borel1975}, Borel computed the  groups  $H_k(\mathrm{SL}(n,\mathbb{Z}); \mathbb{Q})$, for $n>>k$. Stable homology groups of mapping class groups were computed more recently by Galatius \cite{G2004} using the work of Madsen and Weiss \cite{MW2007}. Stability results for non-orientable surfaces were obtained by Wahl \cite{Wahl-homo}, both in terms of genus and of number of boundary components. Her proof is simpler than that of Harer, and it is also based on the action of the mapping class group on complexes of arcs and of curves. We refer the reader to the survey \cite{Ji-Handbook} by Lizhen Ji for a thorough description of  the analogies between mapping class groups and arithmetic groups. There is also an analogous theory for the stability of the homology for automorphism groups of free groups, which was worked out by Hatcher in 
\cite{Hatcher1995}. Let us note that the group $\mathrm{SL}(2,\mathbb{Z})$ belongs to the three classes of groups mentioned: it is the mapping class group of the torus, it is the outer automorphism group of the free group $F_2$ on two generators, and it is an arithmetic group. 

All the analogies and parallels between mapping class groups and other groups show how mathematics is connected.

 In this paper, we survey some rigidity results on actions of mapping class groups, and we try to give some motivation for them.  The content is as follows.

 In \S 2, after a short introduction to rigidity, we review rigidity theorems concerning mapping class group actions on several spaces. The results say that except for some special surfaces, the homomorphism induced by the action from the mapping class group into the automorphism group of the structure with which the space is equipped is an isomorphism.

In \S 3, we review in some detail four theorems on rigid actions of the mapping class group. The reason for which we present these theorems together is that they have several common features. First of all, they concern actions on spaces of foliations. Secondly, the corresponding rigidity results concern \emph{homeomorphism} groups except in the first case, where the homeomorphism is assumed to be piecewise-linear. Homeomophism rigidity theorems are rare in this context, and to our knowledge there are the only known ones. Thirdly, all the spaces in these examples can be seen as boundary spaces. The fourth common feature, which is the most important one, concerns the proofs of these results. In all four cases, it is the ``badness" of the space (non-linearity and non-Hausdorffness), and a hierarchy on this badness, which allows us to get our hand on the set of systems of simple closed curves which is naturally embedded in the space. This subset is then shown to be  invariant by the automorphisms of the structure, with the number of components of the curves preserved. From this, one gets an action of a homeomorphism of the given space on the curve complex, and applying the theorem of Ivanov-Korkmaz-Luo on the rigidity of this action (Theorem \ref{th:IKL} below) shows that the homeomorphism we started with is induced by an element of the mapping class group. 

In \S 4, we briefly review some other contexts in which mapping class groups act, and we mention some questions.
       
            I would like to thank Olivier Guichard and Lizhen Ji who made comments on an early version of this paper.
            
\section{Rigidity and actions of mapping class groups}\label{s:rigidity}

  We start with some remarks concerning rigidity.
 
There are rigidity results in various mathematical contexts which are analogous in spirit to the rigidity results that we present below, in the sense that they say that the automorphisms of a certain structure arise from geometric transformations of the underlying space. One example is the following result for  linear maps between Banach spaces:

  \begin{theorem}[Mazur and Ulam, \cite{MU} 1932]  For any normed vector spaces $V$ and $V'$, any isometry $T:V\to V'$ which is distance-preserving and which preserves the origins is linear.
  \end{theorem}
 
Thus, an isometry between normed spaces arises necessarily from the linear structure of the underlying vector spaces, with no norms involved.
  See Banach \cite{Banach1932} Chapter 7 \S 4 (p. 166) for a poof.   There is also a generalization to approximate isometries, see \cite{BS}.
    
    Another characteristic example in the context of Banach spaces is the following result on the isometries of $L^p$ spaces:
    
\begin{theorem}[Banach 1932, see Lamperti  \cite{Lamperti} for generalizations] 
 Let $1\leq p<\infty$ and $p\not=2$, and let $T$ be an isometry of the metric space $L^p([0,1],\mathbb{R})$.
   Then, there exist a  Borel measurable mapping $\phi$ of $[0,1]$ onto itself which is surjective (up to measure zero)
   and an element $h\in L^p([0,1],\mathbb{R})$ such that for every $f$ in $L^p([0,1],\mathbb{R})$ we have $T(f)(x)= h(x)f(\phi (x))$.  
   \end{theorem}
     
  This says that an isometry of the $L^p$ space arises necessarily from a transformation of the underlying measure spaces. The theorem is also discussed in the paper \cite{Nowak} in Volume II of this series.

     There are other such rigidity results for isometries between Banach spaces, see in particular Chapter 11 \S 5 of \cite{Banach1932}.

 The following is a classical rigidity result, see \cite{Thurston1997} p. 63:
 
 \begin{theorem} 
    Any diffeomorphism of hyperbolic space that takes all hyperbolic lines to hyperbolic lines is an isometry.
    \end{theorem}  
Another example of a rigidity result is the following classical theorem of Mostow (\cite{mos2}  \cite{Mostow1973}):

\begin{theorem} 
Any isomorphism between fundamental groups of complete finite-volume hyperbolic manifolds of the same dimension $\geq 3$ is induced by a unique isometry between the two manifolds.
\end{theorem}  

Tha above form of Mostow rigidity is well known to low-dimensional topologists and geometers because a proof is contained in Thurston's Princeton famous notes \cite{Thurston1979}.  Mostow proved more general version which holds for locally symmetric spaces of finite volume. There are also stonger versions due to Margulis and Prasad. See the discussion in the paper \cite{LJ} in this Volume.

In another direction, Mostow's rigidity theorem has a wide generalization to a rigidity theorem concerning hyperbolic 3-manifolds. The result is known as the Ending Lemination Conjecture, formulated  by Thurston and it proved by Brock, Canary and Minsky, see \cite{Minsky2010} \cite{BCM2012}. The result says that if two hyperbolic 3-manifolds are homeomorphic by a homeomorphism which preserves the so-called end invariants, then they are isometric. We shall mention this results and the end invariants several times in what follows.

   Let us now recall a well-known rigidity result for the action of the mapping class group.   
 
 We shall call a {\it curve} on $S$ is a submanifold of the interior of $S$ which is homeomorphic to a circle.
   A curve is said to be {\it essential}  if it is not homotopic to a point or to a boundary component of $S$. The {\it curve complex}  $C(S)$ is the simplicial complex whose  $k$-simplices, for all $k\geq 0$, are the $k+1$ isotopy classes of essential pairwise non-isotopic and pairwise disjoint essential curves. The extended mapping class group acts simplicially on $C(S)$.
 
The curve complex was introduced by Harvey in 1978 \cite{Harvey1981}, with the idea that this complex encodes some boundary structure for Teichm\"uller or for Riemann's moduli space, in analogy with Tits buildings which encode boundary structure of symmetric spaces and Lie groups. 

 The following rigidity result concerning the curve complex says that except for some special cases, any automorphism of the curve complex is induced by a homeomorphism of the underlying surface. 
 
 \begin{theorem}[Ivanov-Korkmaz-Luo  \cite{ivanov1997} \cite{korkmaz} \cite{Luo2000}]\label{th:IKL}  Consider a surface $S_{g,n}$ which is not  a sphere with at most four holes or a torus with at most one hole. 
Then:
\begin{enumerate} 
\item  For $(g,n)\not\in\{(1,2),(2,0)\}$, the natural homomorphism $$\Gamma^*(S_{g,n})\to \mathrm{Aut}(C(S_{g,n}))$$ is an isomorphism.

\item  The natural homomorphism $\Gamma^*(S_{2,0})\to \mathrm{Aut}(C(S_{2,0}))$ is surjective and its kernel is of order two, generated by the hyperelliptic involution.

\item\label{LLuo} The natural homomorphism $\Gamma^*(S_{1,2})\to  \mathrm{Aut}(C(S_{1,2}))$ is neither surjective nor injective. The kernel of this homomorphism is a subgroup of index 2 of $\Gamma^*(S_{1,2})$ generated by the hyperelliptic involution and its image is a subgroup of index 5 in $ \mathrm{Aut}(C(S_{1,2}))$. The image consists in the simplicial automorphisms of 
$C(S_{1,2})$ that preserve  the set of vertices represented by nonseparating curves.

\end{enumerate}
\end{theorem}

The result has often been compared to the rigidity result of Tits saying that any simplicial automorphism of an irreducible thick spherical building of rank $\geq 2$ associated to a linear algebraic group $G$ is induced by an automorphism of $G$, see \cite{Tits1974} and the survey \cite{Ji-Handbook}. 

As already mentioned in the introduction, the curve complex was used to obtain results which are analogous to results about lattices in Lie groups which were proved using Tits buildings. It is in this sense that the curve complex plays for the mapping class group the role played by Tits buildings for arithmetic groups. For instance, the virtual cohomological dimension of arithmetic groups was computed by Borel and Serre using Tits buildings \cite{BS1973}; likewise, the virtual cohomological dimension of the mapping class group was computed by Harer using the curve complex \cite{Harer1986}.

 As an application of Theorem \ref{th:IKL}, Ivanov gave in \cite{ivanov2011} a geometric proof of Royden's theorem that we recall below on the rigidity of the action of the mapping class group on Teichm\"uller space  \cite{Royden1971}. The new proof is global and geometric as opposed to the original proof by Royden which is local and analytic. Ivanov's proof is also in the spirit of the proof of the result of Tits which we mentioned about simplicial automorphisms of spherical buildings. 
Theorem \ref{th:IKL} has several other applications, and the proofs of most of the rigidity results that we mention now on mapping class group actions use this theorem.

    The following is a list of known actions of the (extended) mapping class group; most of them induce isomorphisms between this group and automorphism groups of structures involved. They are classified by type:

           \begin{enumerate}

\item Actions of an algebraic nature: The action on homology, the action on the outer automorphism group of the fundamental group of the surface \cite{Nielsen1927} \cite{Baer1928}, the action  by inner automorphisms on the mapping class group itself  \cite{ivanov0} \cite{mcc3}, and 
the action by inner automorphisms on the Torelli group \cite{FI2005} \cite{MV}.

\item \label{22} Actions of a combinatorial nature: The actions on the curve complex \cite{ivanov1997} \cite{korkmaz} \cite{Luo2000}, on the pants decomposition complex \cite{Margalit2004}, 
on the arc complex \cite{IM} \cite{DDD} \cite{Disarlo-r}, 
on the arc and curve complex \cite{KP1},
on the ideal triangulation graph \cite{KP2},
on the Schmutz graph of nonseparating curves and on the systolic curve complex \cite{Schmutz}, 
on the complex of nonseparating curves \cite{irmak2},
on the Hatcher-Thurston complex of cut systems \cite{IK},
on the complex of separating curves  \cite{brendlemargalit} \cite{Kida-Torelli},
on the Torelli complex \cite{FI2005} \cite{brendlemargalit} \cite{Kida-Torelli},
 on the complex of domains \cite{MCP},
on the disc complex and on other complexes of subsurfaces in 3-manifolds \cite{KoS} \cite{CPT2} \cite{CPT2} \cite{CPT2}, and
on the curve complex of a non-orientable surface \cite{AK}.

\item Holomorphic actions: 
  The actions on the universal Teichm\"uller curve (see \cite{T1944} \cite{Grothendieck-CartanI} and the surveys \cite{Teich-Commentary}  and \cite{AJP1} \cite{JP}),
 on spaces of quadratic differentials \cite{Royden1971}  \cite{EK1974} \cite{La1997} \cite{Markovic2003}, and
 on Teichm\"uller spaces  \cite{Royden1971}  \cite{Earle-Kra}  \cite{EK1974} \cite{EG1996} \cite{Markovic2003}.

\item Actions by isometries:
 Isometries of the Teichm\"uller metric \cite{Royden1971}  \cite{EK1974},
of the Weil-Petersson metric \cite{MW2002}, and
    of Thurston's metric \cite{Walsh-Hand}.

\item Actions by projective piecewise-linear transformations: The action on the projectivized space of measured foliations  \cite{Papadopoulos2008}.

\item Actions by automorphisms which preserve intersection functions  \cite{Luo1998} \cite{Luo1998b}.

\item \label{s:ho} Actions by homeomorphisms: Homeomorphisms of the space of unmeasured foliations \cite{Ohshika2012}, of the space of geodesic laminations  \cite{ChPP} and of the reduced Bers boundary \cite{Ohshika2011}.
           
        \item Actions on triangulated Teichm\"uller spaces and on decorated Teichm\"uller spaces \cite{Penner1992}  \cite{Penner1993} \cite{Penner} \cite{Penner2012}.
        
   \end{enumerate}

  Several of the above actions are \emph{mixed} actions in the sense that they belong to several groups at the same time. For instance, the combinatorial actions are also isometric (they preserve the natural combinatorial metrics). Likewise, the actions on Teichm\"uller space are isometric and holomorphic at the same time, and they are rigid in both senses. Note however that although most of the above actions are by homeomorphisms, they are not rigid as such. The only rigid homeomorphic actions that we know are those mentioned in \ref{s:ho}.

An analysis of all these actions is contained in \cite{P-Aarhus}.
 Let us make a few comments on some of them.
   
The action of the mapping class group on the homology $H_1(S,\mathbb{Z})\simeq \mathbb{Z}^{2g}$ is rigid only in the sense that every symplectic automorphism of the homology is induced by an element of the mapping class group. The action on homology has a nontrivial kernel, which is the famous Torelli group, a subject of very active research. One of the most natural questions concerning this group is whether it is finitely generated. This question is open for  a closed surface of genus $\geq 3$ and known to be false for genus $2$.

The mapping class group actions on the combinatorial complexes mentioned in (\ref{22}) have several applications and we already mentioned  that Harer used the action on the arc complex and the curve complexes to obtain his homology stability theorem. Hatcher and Thurston used the action on the Hatcher-Thurston complex to obtain a finite presentation of the mapping class group \cite{HT}.  Harer used this presentation to compute the second rational homology group of the mapping class group. The Borel construction for group actions provides spectral sequences which relate the homology of the mapping class group and the homology of the simplex stabilizers of these actions, and thus it reduces the study of mapping class groups to mapping class groups of simpler surfaces and the results are then obtained inductively.
    
Harer used the nonseparating curve complex in \cite{Harer1983} to prove that if $\Gamma_g$ is the mapping class group of a closed orientable surface of genus $g\geq 3$, we have $H_2(\Gamma_g;\mathbb{Q})\simeq \mathbb{Q}$. This result was improved recently by Putman to include the case of the second rational homology of any finite-index subgroup of the mapping class group, for $g\geq 5$ \cite{Putman2011}.

  The introduction of the systolic curve complex \cite{Schmutz}  was motivated by the need for a better understanding of the structure of the collection of systoles on hyperbolic surfaces. In 1986, Thurston outlined the construction a spine of Teichm\"uller space using systoles \cite{Thurston-spine}.

There are several other actions of interest, namely those related to infinite-type surfaces, non-orientable surfaces, the soleniod, the universal Teichm\"uller spaces, and there is also a profinite theory (the so-called Grothendieck-Teichm\"uller theory). There are actions on representations  of fundamental groups of surfaces into various Lie groups (the so-called higher Teichm\"uller theory), etc. but most of the rigidity results in these settings are still conjectural and there is a lot of work that can be done in this field. We shall dwell on this in \S \ref{s:perspectives} below.

The reader might be interested in the following problems related to the above actions:
\begin{problems}
\item Formulate and prove \emph{local} rigidity results for each of the above problems.
\item Formulate and prove rigidity results for \emph{immersions} (that is, locally injective structure-preserving maps) instead of isomorphisms.
\item Formulate and prove rigidity results for maps between spaces associated to non-homeomorphic surfaces.
\end{problems}

For some of the actions described aboven there is some work already done in this direction. the reader can consult the monograph \cite{P-Aarhus}.

  \section{Actions on foliations and laminations}
In this section, we study the following actions of the mapping class group: 
\begin{enumerate}
\item  The action by projective piecewise-linear homeomorphisms on the space  $\mathcal{PMF}$ of projective equivalence classes of measured foliations. 

 \item The action by homemorphisms on the space $\mathcal{UMF}$ of unmeasured foliations, that is, the quotient of  $\mathcal{PMF}$ by the equivalence relation which identifies two elements whenever they are represented by measured foliations that are topologically the same. 
 
  \item The action by homemorphisms on the reduced Bers boundary of Teichm\"uller space.  This space can be seen as a subspace of unmeasured foliation space, but, as Ohshika showed, the topology is not the induced topology. 
  
 \item The action by homemorphisms on the space $\mathcal{GL}$ of geodesic laminations on a hyperbolic surface. The space is equipped with the geometric topology (also called the Thurston topology). Note that there is no identification between $\mathcal{GL}$ and the space of unmeasured foliations (or laminations) on the surface; the former is strictly is larger.  

  \end{enumerate}
      
We describe in some detail the spaces in the above list and we state the corresponding rigidity results for the mapping class group actions on them. 
   In each case, we shall say a few words on how the spaces appear in different contexts.

    \subsection[Automorphisms of the PPL structure]{Automorphisms of the PPL structure of $\mathcal{PMF}$}\label{sectionPL}
   
In this subsection, we consider the action of the mapping class group on Thurston's sphere $\mathcal{PMF}$ of projective equivalence classes of measured foliations of a surface. This space carries a natural projective piecewise-linear structure. The union of Teichm\"uller space  with this sphere is a beautiful example of a compact space homeomorphic to a closed ball whose interior carries a natural smooth structure (in fact, it carries a natural complex-analytic structure) and whose boundary carries a natural projective piecewise-linear structure which is far from being smooth. It was constructed by Thurston as a natural compactification of Teichm\"uller space, and it plays a fundamental role in the proof he gave of his classification theorem of mapping classes.

   We shall recall the basic elements of Thurston's piecewise-linear (PL) structure of $\mathcal{MF}$ defined by train tracks and the quotient projective piecewise-linear (PPL) structure of $\mathcal{PMF}$.
   
We note right away that the PL atlas defined by train tracks is not a maximal atlas in the sense of  analytic structures; in fact, one can easily see that if we take the \emph{maximal} atlas associated to this PL structure, we get a PL structure whose automorphism group is uncountable, and therefore it cannot be the mapping class group.

We first recall the definition of measured foliations and of the space $\mathcal{MF}$. 
\begin{figure}[!hbp]
 \begin{center}
 \scalebox{0.60}{\includegraphics{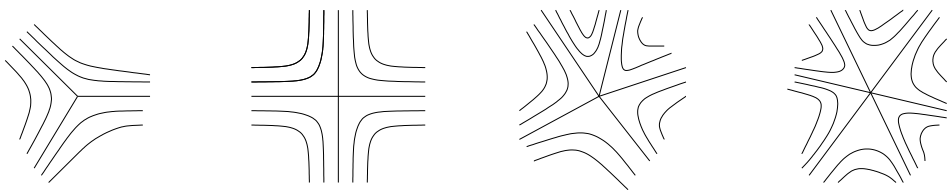}}
 \end{center}
  \caption{Singular points of a measured foliation at an interior point.}
 \label{fig:2_5}\end{figure}
 
We consider foliations on $S$ with singular points  which are of the type given in Figure \ref{fig:2_5}.
 For singular points on the boundary, the local model is such that if we double the surface along its boundary components, we get allowable singularities (as interior singular points). Examples of singular points at the boundary are represented in Figure \ref{boundaryfoliation}.
 \begin{figure}[!hbp]
\centering
\includegraphics[width=.4\linewidth]{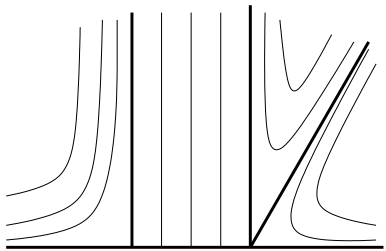}
\caption{\small{Singular points of a measured foliation at a boundary point.}}
\label{boundaryfoliation}
\end{figure}

A {\it transverse measure} for a foliation is a measure on each transverse arc which is equivalent to the Lebesgue measure of a compact interval of $\mathbb{R}$ such that these measures are invariant by the local holonomy maps, that is, the isotopies of arcs that keep each point on the same leaf.

 A {\it Whitehead move} between two measured foliations is the operation of contracting to a point a compact leaf that joins two singular points, or the inverse operation. The move is well-defined up to an ambient isotopy. An example of a Whitehead move is represented in Figure \ref{Whitehead}.
\begin{figure}[!hbp]
 \begin{center}
 \scalebox{0.55}{\includegraphics{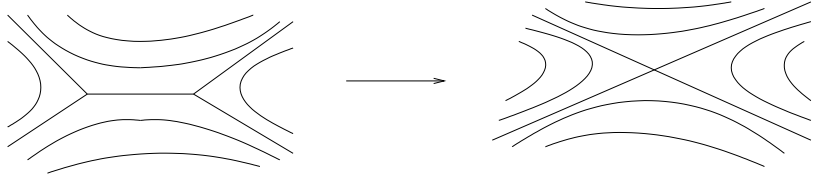}}
 \end{center}
  \caption{Whitehead move: collapsing (or creating) 
   an arc joining two singular points.}
   \label{Whitehead}
 \end{figure}
 
The equivalence relation between measured foliations is generated by ambient isotopy and Whitehead moves that preserve transverse measure.

Given a measured foliations $F$ on $S$, we shall use the notation $[F]$ for its equivalence class in $\mathcal{MF}$, and sometimes also $F$ (without the brackets) when no confusion is possible.

The space of equivalence classes of measured foliations is called {\it measured foliation space}, and it is denoted by 
  $\mathcal{MF}$ or $\mathcal{MF}(S)$. It is equipped with a topology which is defined via its embedding in the space $\mathbb{R}^{\mathcal{S}}_+$ of nonnegative functions on the set  $\mathcal{S}$ of isotopy classes of essential simple closed curves on $S$. This is done using the geometric intersection function which we recall now.
     
Given a measured foliation $F$ on $S$ and an element $\gamma$ in $\mathcal{S}$,  $i(F,\gamma)$ is the infimum of the total measure of $c$ over all closed curves $c$ in the homotopy class $\gamma$ which are concatenations of arcs transverse to $F$ and arcs contained in the leaves of $F$. The total measure of $c$ is the sum of the measures of the subarcs transverse to $F$.

The value $i(F,\gamma)$ does not depend on the choice of the equivalence class of $F$. We obtain in this way  a map 
\[i(.,.) : \mathcal{MF}\times \mathcal{S}\to \mathbb{R}_+
\]
called the {\it geometric intersection function}, or, in short, \emph{intersection function}.\index{intersection function}

Using this function, we obtain a natural map 
\[ F\mapsto  i(F,.)
\]from $\mathcal{MF}$ to $\mathbb{R}_+^{\mathcal{S}}$.
This map is injective, and $\mathcal{MF}$ inherits a topology from this injection, the space $\mathbb{R}_+^{\mathcal{S}}$ being equipped with the product topology. Equipped with this topology, $\mathcal{MF}$ is homeomorphic to $\mathbb{R}^{6g-6}\setminus\{0\}$.  All these results are due to Thurston \cite{Thurston1988} and they are proved in \cite{FLP}. It is sometimes practical to consider the empty foliation as an element of $\mathcal{MF}$ and in this case the space becomes homomorphic to $\mathbb{R}^{6g-6}$.

There is a natural action of the group of positive reals $\mathbb{R}_+^*$ on $\mathcal{MF}$ obtained from the action of $\mathbb{R}_+^*$ on measured foliations defined by multiplying the transverse measure by a constant positive factor. The quotient space of $\mathcal{MF}$ by this action is the \emph{projective foliation space}, denoted by $\mathcal{PMF}$.

From the embedding of  $\mathcal{MF}$ into the space $\mathbb{R}_+^{\mathcal{S}}$ we obtain an embedding of $\mathcal{PMF}$ in the projective space $\mathbb{P}\mathbb{R}_+^{\mathcal{S}}$.

It will be convenient in this subsection and in the next one to represent an element of $\mathcal{MF}$ by a {\it partial} measured foliation. This is a measured foliation whose support, $\mathrm{Supp}(F)$,  is a nonempty (and not necessarily connected) subsurface $S'$ with boundary of $S$ such that the boundary curves of $S'$ are all essential in $S$. The partial foliation may have singularities at the boundary of $S'$, but in the complement of these singular points the foliation is tangent to the boundary of $S'$. Any partial measured foliation on $S$ gives a well-defined element of $\mathcal{MF}$, which is the equivalence class of the measured foliation obtained by collapsing each non-foliated region onto a spine, cf. \cite{FLP}.

We now recall the theory of train tracks which is at the basis of the piecewise-linear theory of $\mathcal{MF}$. 

A {\it train track} $\tau$ on $S$ is a graph embedded in $S$ whose vertices we shall suppose trivalent (although this is not necessary) and such that at any vertex, the three half-edges that abut have a well-defined tangent at that point; that is, we have a notion of two half-edges abutting  from one side and of the remaining half-edge abutting from the other side. (In fact, the use of the word ``tangent" here can be somehow misleading, since no smooth structure on $S$ is needed in the theory. Instead, one can simply say that there is a well-defined notion of ``two half-edges abutting from one side" and ``one half-edge abutting from the other side" at each vertex. Equivalently, we just single out one of the three half-edges abutting at each vertex.) 

 The local structure at a vertex is represented in Figure \ref{fig:switch}.

A vertex of $\tau$ is also called a {\it switch}. We shall call a {\it corner} of $S$ a region in a neighborhood of a switch which is contained between the two half-edges that abut from the same side.

   \begin{figure}[!hbp]
 \begin{center}
\scalebox{0.6}{\includegraphics{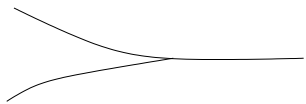}}
\end{center}
\caption{The local model at a switch.}
\label{fig:switch}
\end{figure}

An \emph{edge} of a train track is a connected component of the complement of the set of switches. We note that a train track may have no vertices at all, i.e. it could consist of a union of disjoint simple closed curves; in that case we shall also call such a closed curve an \emph{edge} of the train track.

On a surface with boundary, a train track may be either tangent or transverse to the boundary.

On closed surfaces, all the train tracks $\tau$  that we consider satisfy the following condition:  no  connected component of $S\setminus\tau$ is a disc with $0$, $1$ or $2$ corners or an annulus with no corner (see Figure \ref{excluded}). We shall call such a train track \emph{admissible} whenever this property needs to be referred to (although all our train tracks will be admissible).
\begin{figure}[!hbp]
 \begin{center}
\scalebox{0.5}{\includegraphics{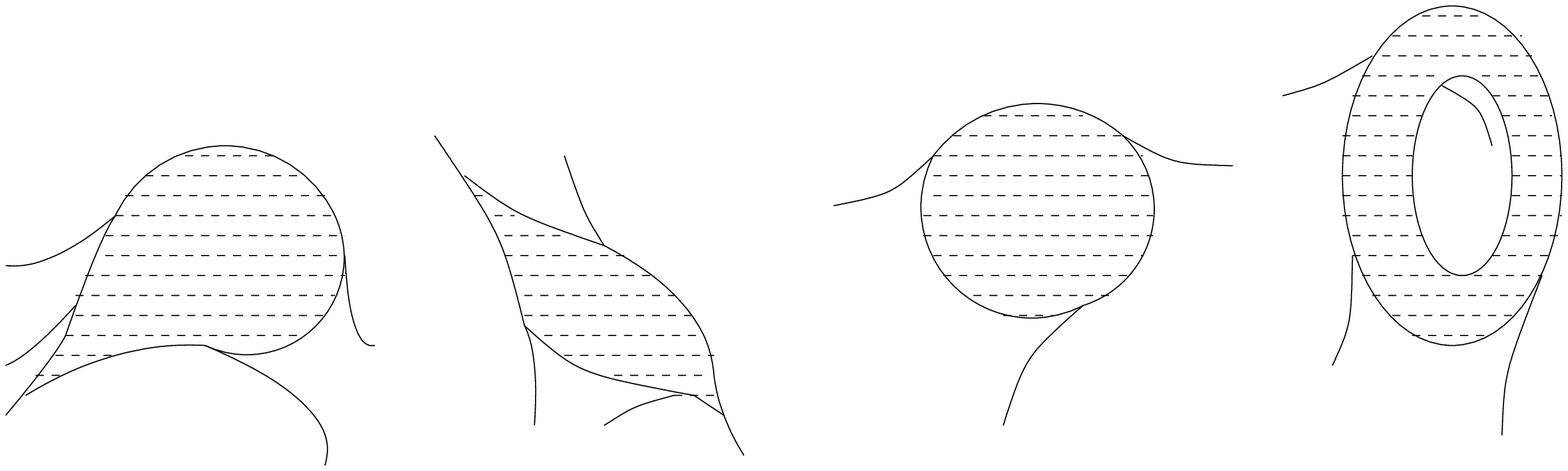}}
\end{center}
\caption{Excluded complementary components, for the complement of a train track.}
\label{excluded}
\end{figure}
   For surfaces with boundary, the excluded complementary components are those such that they give rise to excluded complementary components on the closed surfaces obtained by doubling the surfaces with boundary.

Train track theory works for general surfaces of finite type, with punctures and/or boundary components, but in order to avoid several technicalities, we shall deal from now on with the case where $S$ is a closed surface.

A train track $\tau$ is said to be {\it maximal} if every component of $S\setminus \tau$ is a disc with three corners on its boundary. 

 There are three basic moves that can be performed on a train track and they are represented in Figure \ref{shift-split}. They are called a \emph{shift}, a \emph{right split} and a \emph{left split}. These three moves will be essential in what follows.
   \begin{figure}[!hbp]
 \begin{center}
\scalebox{0.5}{\includegraphics{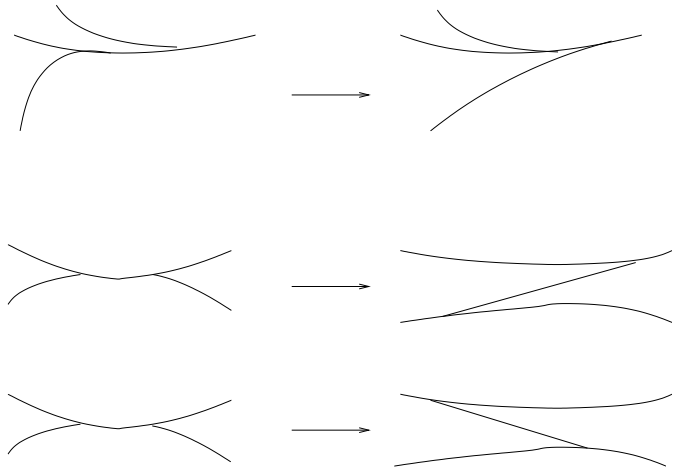}}
\end{center}
\caption{From top to bottom, the three elementary moves on a train track: a shift, a left split and a right split.}
\label{shift-split}
\end{figure}

 Any train track $\tau$ has a {\it regular neighborhood} $N(\tau)$ foliated by arcs that are called the {\it ties}. The local picture of the foliation by the ties near a switch is represented in Figure \ref{fig:ties}. The regular neighborhood $N(\tau)$, equipped with its foliation by the ties, is well-defined up to isotopy, and there is a natural projection $N(\tau)\searrow \tau$ from $N(\tau)$ to a train track (isotopic to) $\tau$ defined by collapsing each tie to a point.
  \begin{figure}[!hbp]
 \begin{center}
\scalebox{0.5}{\includegraphics{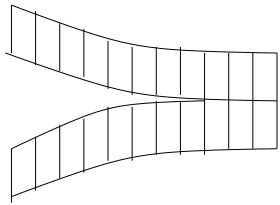}}
\end{center}
\caption{The regular neighborhood of a train track and the local structure of the ties near a switch.}
\label{fig:ties}
\end{figure}

Let $\tau$ be a train track on $S$ and  let $e_1,\ldots,e_N$ be its edges. Let $\mathbb{R}^N$ denote the real vector space with basis $\{e_1,\ldots,e_N\}$ and let $(x_1,\ldots,x_N)$ denote the coordinates of a point in that space. We let $V_{\tau}\subset \mathbb{R}^N$ be the closed convex cone in $\mathbb{R}^N$ defined by the system of equations
  \[\displaystyle   
  \begin{cases}
x_i\geq 0 & \text{ for every\ } i=1,\ldots, N\\
x_i=x_j+x_k  &\text{ for every switch of\  } \tau
 \end{cases}
\]
where, in the last equation, $x_j$ and $x_k$ denote the coordinates at the two edges that abut from the same side at the given switch, and $x_i$  the coordinate on the edge that abuts from the other side.
 
 The coordinates $x_i$ on the edges are also called the \emph{weights} of the edge.
 
 A train track $\tau$ is said to be {\it recurrent} if there exists an element $(x_1,\ldots,x_N)$ of $V_{\tau}$ satisfying $x_i>0$ for all $i=1,\ldots,N$.

For any train track $\tau$ on $S$, there is a map 
\begin{equation}\label{eq:var}
\varphi_{\tau}:V_{\tau}\to\mathcal{MF}
\end{equation}
 defined as follows:

 Consider the regular neighborhood $N(\tau)$ of $\tau$ equipped with its projection $N(\tau)\searrow \tau$ and let $(x_1,\ldots,x_N)$ be a nonzero element of $V_{\tau}$. For each nonzero coordinate $x_i$ ($i\in \{1,\ldots,N\}$), consider the inverse image of the edge $e_i$ by the projection  $N(\tau)\searrow \tau$. The closure of the interior of this inverse image has a natural structure of a rectangle equipped with a foliation induced by the ties, which we call the \emph{vertical} foliation. We equip this rectangle with another foliation, which we  call the \emph{horizontal} foliation, whose leaves are segments which are transverse to the ties and which join the two vertical edges of the rectangle,  that is, the two edges that consist of ties. We equip this horizontal foliation with a transverse measure whose total mass is equal to $x_i$.  
 
 Finally, we glue the various foliated rectangles so obtained along their vertical sides, by using measure-preserving homeomorphisms. We obtain a partial measured foliation on $S$, which gives a well-defined element of $\mathcal{MF}$. 
The zero element of $V_{\tau}$ is sent to the empty foliation of $\mathcal{MF}$. 
This defines the map $\varphi_{\tau}:V_{\tau}\to\mathcal{MF}$.

 The notion of train track was introduced by Thurston in \cite{Thurston1979}. A standard reference on train tracks is the monograph \cite{PH} by Penner and Harer. 
The following result is outlined by Thurston in his Notes \cite{Thurston1979}. A proof is contained in \cite{P1} p. 20 and ff.
 
\begin{theorem}\label{homeoF}
The map $\varphi_{\tau}$ is a homeomorphism onto its image. In the case where $\tau$ is maximal and recurrent,  the image $\varphi_{\tau}(V_{\tau})$ has nonempty interior in $\mathcal{MF}$.
\end{theorem}

A measured foliation $F$ (or its equivalence class $[F] \in\mathcal{MF}$) is said to be {\it carried by} a train track $\tau$ if $[F]$ is in the image set of $V_{\tau}$ by the map  $\varphi_{\tau}$ in (\ref{eq:var}).

  We need to recall the definition of a PL map.

Let $M$ be a positive integer. A \emph{(closed) linear half-space}\index{linear half space} in $\mathbb{R}^M$ is (the closure of) a connected component of the complement of a hyperplane. We shall say that a subset $V$ of $\mathbb{R}^M$ is a {\it linear polytope}  if it is the intersection of a finite number of closed linear half-spaces of $\mathbb{R}^M$. 
Let us stress on the fact that we are talking about linear and not only affine half-spaces, so that a linear polytope $V$ is $\mathbb{R}_+^*$-homogeneous (invariant by multiplication by elements of $\mathbb{R}_+^*$). In particular, it is noncompact (unless it is empty). All the polytopes we consider are linear. With this in mind, we shall henceforth use the name {\it polytope}\index{polytope} to denote a linear polytope.

The dimension of a polytope $V\subset \mathbb{R}^M$ is the dimension of the smallest (in the sense of inclusion) linear subspace of $\mathbb{R}^M$ that contains it.
It is also the smallest dimension of a vector space in which $V$ embeds.

Note that a polytope is closed and convex.

The {\it relative interior}\index{relative interior}\index{polytope!relative interior} of a polytope $V$ is the topological interior of $V\cap A$ in $A$, with $A$ being the smallest linear subspace of $\mathbb{R}^M$ containing $V$.

Let $M$ and $N$ be two positive integers, let $V\subset \mathbb{R}^M$ be a finite union of polytopes of the same dimension, $V_1,\ldots, V_n\subset \mathbb{R}^M$, and let $f:V\to \mathbb{R}^N$ be a map.
We shall say that $f$ is PL\index{PL map} relative to the decomposition $V=V_1\cup \ldots\cup V_n$
 if $f$ is continuous and if the restriction of $f$ to the relative interior of each of the polytopes $V_1,\ldots, V_n$ is the restriction of a linear map from $\mathbb{R}^M$ to $\mathbb{R}^N$.

Given a subset $V$ of $\mathbb{R}^M$, we shall say that a map $f:V\to \mathbb{R}^N$ is piecewise-linear (PL) if it is PL relatively to some decomposition of $V$ as a union of polytopes $V_1,\ldots, V_n$.

 For each maximal recurrent train track $\tau$, we let $U_\tau=\varphi_\tau (V_\tau) \subset \mathcal{MF}$ be the image of the associated map $\varphi_{\tau}:V_{\tau}\to\mathcal{MF}$ and we let $\psi_{\tau}:U_{\tau}\to V_{\tau}$ denote the
  inverse of the homeomorphism $\varphi_{\tau}:V_{\tau}\to \varphi(V_{\tau})=U_{\tau}$ of Theorem \ref{homeoF}.
  
Let $\mathcal{A}$ be the following collection of maps:
 $$\mathcal{A}=\{(U_{\tau},\psi_{\tau})\ \vert \tau \text{ is a maximal recurrent train track }\}.$$

\begin{theorem}[Thurston] \label{th:Th-PL} The set $\mathcal{A}$ is an atlas of a PL structure on $\mathcal{MF}$. 
\end{theorem}

The meaning of this statement is that the coordinate changes of this atlas are PL.

 We shall prove a rigidity theorem (Theorem  \ref{th:autPL} below) for the automorphism group of this structure, which we call the \emph{train track PL structure of $\mathcal{MF}$}. The proof of this theorem uses a precise description of the coordinate changes $\psi_{\tau^{-1}\sigma}=\psi_{\sigma}\circ\psi_{\tau}^{-1}$ of the atlas $\mathcal{A}$,  each such coordinate change map being defined on an appropriate subset of $V_{\tau}$. We shall recall here some details given in Chapter 1 of \cite{P1}, because we shall use them in the proof of Theorem  \ref{th:autPL}.

 If $\tau$ and $\sigma$ are two train tracks on $S$, we say that $\tau$ is {\it carried}by $\sigma$, and we write this relation as $\tau\prec\sigma$, if $\tau$ is isotopic to a train track $\tau'$ which is contained in a regular neighborhood $N(\sigma)$ of $\sigma$ and which is transverse to the ties of $N(\sigma)$. When $\tau\prec\sigma$, there is a natural linear map from the closed convex cone $V_{\tau}$ to the closed convex cone $V_{\sigma}$ which induces the inclusion map at the level of the two subspaces $\varphi_{\tau}(V_{\tau})$ and $\varphi_{\sigma}(V_{\sigma})$ of $\mathcal{MF}$. The map $V_{\tau}\to V_{\sigma}$ is defined by associating to each weight system on the edges of $\tau$ the weigh system on the edges of $\sigma$ where the weight on an edge $e$ is defined as the sum of the edges of $\tau$ (counted with multiplicity) that traverse a tie which is an inverse image of $e$ with respect to the map $N(\tau)\searrow \tau$. (The result is independent of the choice of the tie).

 The following definition and proposition are needed in the proof of Theorem  \ref{th:autPL}.

\begin{definition}
 Let $\tau$ be a maximal recurrent train track on $S$ and let $T=\{\tau_1,\ldots,\tau_n\}$ be a family of maximal recurrent train tracks on $S$.
 We say that the family $T$ is
{\it adapted to $\tau$} if the following properties hold:   
 
 \begin{enumerate}

 \item \label{a1} For each $i=1,\ldots,n$, $\tau_i \prec\tau$.
\item   \label{a4} For every $i$ and $j$ satisfying $1\leq i<j\leq n$, the interiors of $U_{\tau_{i}}$ and  $U_{\tau_{j}}$ are disjoint. 
 \item $\displaystyle \cup_{i=1}^n U_{\tau_{i}}=U_\tau$

 \end{enumerate} 
\end{definition}
 \begin{proposition}\label{prop:change}
 Let $(U_{\tau},\psi_{\tau})$ and $(U_{\sigma},\psi_{\sigma})$ be two coordinate charts in $\mathcal{A}$ 
 and let $\psi_{\tau^{-1}\sigma}$ be the corresponding coordinate change, 
defined on the subset $\psi_{\tau}(U_{\tau}\cap U_{\sigma})$ of $V_{\tau}$. For every point $F$ in the interior of $\psi_{\tau}(U_{\tau}\cap U_{\sigma})$, we can find a family  of train tracks $T=\{\tau_1,\ldots,\tau_n\}$ which is adapted to $\tau$ and which furthermore satisfies the following properties:  
 \begin{enumerate}

 \item \label{it1} $\tau_i \prec\sigma \text{ for all\  } i=1,\ldots,n$.
 
 \item The union  $\bigcup_{i=1}^n U_{\tau_{i}}$ is a neighborhood $N(F)$ of $F$ in $\mathcal{MF}$, and 
 $F$ belongs to the set $U_{\tau_{i}}$ for all $i=1,\ldots,n$.
\item \label{ccc} The linear maps that underlie the PL coordinate-change map $\psi_{\tau^{-1}\sigma}$ are induced by the linear maps defined by the relations $\tau_i\prec\tau$ and $\tau_i\prec\sigma$, and the inverses of such maps.

 \end{enumerate}
  \end{proposition}
 Property (\ref{ccc}) of Proposition \ref{prop:change} implies that the restriction of the coordinate change map $\psi_{\tau^{-1}\sigma}$ to the neighborhood $N(F)$ of $F$ is linear on each subset $U_{\tau_{i}}$ of $N(F)$. In fact, $\psi_{\tau^{-1}\sigma}$ is PL relative to the given decomposition. 
 
 Proposition \ref{prop:change} is the basic technical result which implies that the coordinate changes in the atlas $\mathcal{A}$ are piecewise-linear. Without going deeply into the details, let us give an idea of the proof. Full details are given in \cite{P1}, Chapter 1.

The proof of the rigidity theorem is based on an analysis of the structure of the set of singular points of the train track PL functions, that is, the points where $f$ is not linear. Indeed, in the atlas $\mathcal{A}$, there are points where non-linearity occurs, and points where non-linearity cannot occur. There is also a hierarchy on the set of points where non-linearity may occur: there are points where it may be more severe than at others. Let us give right away the reason for that. The change of coordinates, for the PL atlas, are obtained by repeated application of the elementary moves on the train tracks (shifts and splits). These elementary moves define subdivisions of the domains of the charts in $\mathcal{A}$ that parametrize the measured foliations classes. At the level of foliations carried by the train track, an elementary move is performed when the foliation has compact leaves joining singular points. Thus the ``bad sets" of the PL structure (that is, those where the coordinate change is not linear) are classes of measured foliations which have compact leaves joining singular points. And the more there are such compact leaves, the more one can perform subdivisions of the domain of the chart, that is, the larger can be the codimension of the singular set to which this point belongs. By examining the topology of measured foliations, we see that the highest degree of singularity corresponds to the set of equivalence classes of measured foliations that are represented by simple closed curves. Thus, this subset of $PMF$ is preserved by a PL map. For similar reasons, the subset of measured foliations represented by systems of $n$ simple closed curves with $n$ fixed is also preserved. This gives the induced action of a PL homeomorphism of $\mathcal{MF}$ (or $\mathcal{PMF}$) on the curve complex.

Now we give some more definitions which make these ideas more precise.

\begin{definition}[The singular set of a PL function] Let $V$ be a polytope and let $f:V\to \mathbb{R}^N$ be a PL function relative to a finite union of polytopes  $V=V_1\cup\ldots\cup V_n$ in $\mathbb{R}^M$. The \emph{singular set} of $f$, denoted by $\mathrm{Sing}(f)$, is the set of points $x\in V_1\cup\ldots\cup V_n$ such that $f$ is not linear in any neighborhood of $x$.

\end{definition}

We observe the following two facts:

(F1) The  set $\mathrm{Sing}(f)$ is a union of codimension-1 submanifolds of $V$, each of which is equal to the intersection of two sets in the collection of polytopes $\{V_1,\ldots, V_n\}$.

(F2) If the polytopes $V_i$ have dimension $D$, then $\mathrm{Sing}(f)$ is contained in a union of polytopes  of dimension $D-1$ in $\mathbb{R}^M$ such that the restriction of $f$ to $\mathrm{Sing}(f)$ is PL.

Repeating this, we find a nested sequence of subsets of $V$, 
\[
\mathrm{Sing}_0(f)\supset \mathrm{Sing}_1(f)\supset \ldots\mathrm{Sing}_k(f),
\] 
where:
\begin{enumerate}

\item  $ \mathrm{Sing}_0(f)= \mathrm{Sing}(f)$;
 
\item  for every integer $i$ satisfying $2\leq i\leq k$, $\mathrm{Sing}_i(f)$ is the singular set of the restriction of $f$ to  $\mathrm{Sing}_{i-1}(f)$.
 \end{enumerate}

Note that $k\leq n$ and that for each $2\leq i\leq k$, $\mathrm{Sing}_{i}(f)$ is contained in a codimension-1 subset of $\mathrm{Sing}_{i-1}(f)$ (see Observation (F2) above).
 
Let $f$ be a PL function defined on a set $V=V_1\cup\ldots\cup V_n$ as above. The associated sequence $\mathrm{Sing}_0(f)\supset \mathrm{Sing}_1(f) \supset \ldots\mathrm{Sing}_k(f)$ defines a stratification of $V$, each stratum having a well-defined codimension in $V$. We shall call this stratification the {\it flag} induced by $f$ on the set $V$, and we shall denote this flag equipped by its stratification by $\mathrm{Fl}(f)$.

Now we consider more precisely  the PL functions relative to the train track PL structure.
 
  \begin{definition}[Train track PL function]\label{def:TTPL} Let $N$ be a nonnegative integer. A map $f:\mathcal{MF}\to  \mathbb{R}^N$ is said to be a {\it train track $\mathrm{PL}$ map} if for every $x$ in  $\mathcal{MF}$, there exists a chart $(U_{\tau},\psi_{\tau})$ belonging to the atlas  $\mathcal{A}$ mentioned in Theorem \ref{th:Th-PL} such that the set $U_{\tau}$ contains $x$ in its interior and  the map $f\circ \psi^{-1}_{\tau}$ defined on $V_{\tau}=\psi_{\tau}(U_{\tau})$ is PL. Furthermore, we require that there exists a coordinate change map $\psi_{\tau^{-1}\sigma}$ belonging to the atlas $\mathcal{A}$ that has $\psi(x)$ in the interior of its domain, and such that the singular sets of the restrictions of the maps $f\circ \psi^{-1}_{\tau}$ and $\psi_{\tau^{-1}\sigma}$ to the set 
  $V_{\tau}$ coincide in a neighborhood of $\psi_{\tau}(x)$ 
  in $V_{\tau}$. (The last condition says that the train track PL functions are not allowed to have singular points other than those that already appear is the train track coordinate change maps.)

  \end{definition}

  Let $\mathcal{P}$ be the set of  train track PL functions on $\mathcal{MF}$. Due to the last condition in Definition \ref{def:TTPL}, the set $\mathcal{P}$ is the set of  smoothest possible  maps on $\mathcal{MF}$ relatively to the atlas $\mathcal{A}$. 
  
     \begin{definition}[Automorphisms of the train track PL structure] A homeomorphism $h:\mathcal{PMF}\to\mathcal{PMF}$ is an {\it automorphism of the train track PL structure} if $h$ is the quotient map of a homeomorphism $h_0: \mathcal{MF}\to \mathcal{MF}$ which preserves the set $\mathcal{P}$ of train track PL functions. 
   \end{definition}

The homeomorphisms of $\mathcal{PMF}$ that preserve the train track PL structure form a group, which we call the automorphism group of the train track PL structure. We denote this group by $\mathrm{Aut}(\mathcal{PMF},\mathcal{P})$.

\begin{proposition}\label{prop:PPP}
The action on $\mathcal{MF}$ of any element of the extended mapping class group preserves the set of train track PL functions.
\end{proposition}
\begin{proof}
Let $h:\mathcal{MF}\to\mathcal{MF}$ be a homeomorphism induced by an extended mapping class, let $N$ be a positive integer and let $f:\mathcal{MF}\to\mathbb{R}^N$ be a map in $\mathcal{P}$. For each $x$ in $\mathcal{MF}$, let $(U_{\tau},\psi_{\tau})$  and $\psi_{\tau^{-1}\sigma}$ be respectively a chart in $\mathcal{A}$ and a coordinate change map in $\mathcal{A}$ satisfying the properties required in Definition \ref{def:TTPL}. Then,  $\tau'=h(\tau)$ and $\sigma'=h(\sigma)$ are maximal recurrent train tracks on $S$, $(h(U_{\tau}),\psi_{\tau'})$ is a coordinate chart in $\mathcal{A}$, $\psi_{\tau'{}^{-1}\sigma'}$ is a coordinate change in the atlas $\mathcal{A}$, and these data satisfy the properties required in Definition \ref{def:TTPL} with respect to the point $h(x)$ instead of the point $x$. 
\end{proof}
From Proposition \ref{prop:PPP}, we get a homomorphism
\[\Gamma^*(S)\to \mathrm{Aut}(\mathcal{PMF},\mathcal{P}).\]

 Our next goal is to prove that, except for a finite number of special surfaces (like in most of the other rigidity theorems), this homomorphism is an isomorphism (Theorem \ref{th:autPL} below).

Before proving this theorem, we need to establish a few more notation and results.

 A {\it system of curves} on $S$ is the isotopy class of a  collection of disjoint and pairwise non-isotopic essential curves on $S$. For $g\geq 2$, the number of elements in such a collection is bounded above by $3g-3$.

For each $k$ satisfying $1\leq k\leq 3g-3$, let $\mathcal{MF}_k\subset \mathcal{MF}$ be the set of measured foliation classes $x$ satisfying the following:
\begin{enumerate}
\item \label{CC1} There is no chart $(U_{\tau},\psi_{\tau})$ in $\mathcal{A}$ having $x$ in the interior of its domain $U_{\tau}$ such that there exists a coordinate change $\psi_{\tau^{-1}\sigma}$ having $\psi_{\tau}(x)$ in the interior of its domain, with $\psi_{\tau}(x)$ being on a stratum of dimension $\leq k-1$ of the flag $\mathrm{Fl}(\psi_{\tau\sigma})$.
\item \label{CC2} There exists a coordinate chart $(U_{\tau},\psi_{\tau})$ in $\mathcal{A}$ having $x$ in the interior of its domain and a coordinate change $\psi_{\tau^{-1}\sigma}$ having $\psi_{\tau}(x)$ in the interior of its domain, such that $\psi_{\tau}(x)$ is on a stratum of dimension $k$ of the flag $\mathrm{Fl}(\psi_{\tau^{-1}\sigma})$.
\item \label{CC3} The element $\psi_{\tau}(x)$ is a convex combination of $k$ elements in the 1-stratum of $\mathrm{Fl}(\psi_{\tau^{-1}\sigma})$ with respect to the linear structure on $V_{\tau}=\psi_\tau(U_\tau)$ induced from its inclusion in $\mathbb{R}^N$.
\end{enumerate} 
  
In particular, $\mathcal{MF}_1\subset \mathcal{MF}$ is simply the set of measured foliation classes $x$  such that  there exists a coordinate chart $(U_{\tau},\psi_{\tau})$ in $\mathcal{A}$ having $x$ in the interior of its domain and a coordinate change $\psi_{\tau^{-1}\psi}$ having $\psi_{\tau}(x)$ in the interior of its domain such that $x$ is on a stratum of dimension 1 of the flag $\mathrm{Fl}(\psi_{\tau^{-1}\sigma})$ defined by the singular set of the coordinate change.

\begin{proposition}\label{prop:preserves}
For each $k\geq 0$, every element of $(\mathcal{MF},\mathcal{P})$ preserves the set  $\mathcal{MF}_k$.
\end{proposition}
\begin{proof}
An automorphism of $(\mathcal{MF},\mathcal{P})$ acts on the set of flags of the coordinate changes $\psi_{\tau^{-1}\sigma}$ of $\mathcal{A}$, that is, it carries the flag of any coordinate change in $\mathcal{A}$ to a flag of some coordinate change in $\mathcal{A}$, and it preserves Properties (\ref{CC1}) to  (\ref{CC3}) above that define $\mathcal{MF}_k$, for each $k\geq 1$.
\end{proof}

The proof of the rigidity theorem \ref{th:autPL} is based on a geometric characterization of the elements in each set $\mathcal{MF}_k$, which we give now.

Let $\mathcal{S'}$ be the set of homotopy classes of systems of curves on $S$, that is, the set of homotopy classes of unions of disjoint pairwise non-homotopic simple closed curves.

We assume that $g\geq 2$ and, for every integer $k$ satisfying $1\leq k\leq 3g-3$, we denote by $\mathcal{S}_k$ the subset of $\mathcal{S}'$ consisting of isotopy classes of systems of curves of cardinality $k$. 
 
In particular $\mathcal{S}_1=\mathcal{S}$.

Recall that for each integer $k$ satisfying $1\leq k\leq 3g-3$, there is a natural inclusion $j_k:(\mathbb{R}_+^*)^k\times \mathcal{S}_k \hookrightarrow \mathcal{MF}$, defined by associating to each $v\in(\mathbb{R}_+^*)^k$ and to each  element $C\in\mathcal{S}_k$ the equivalence class of a partial measured foliation $F$ which has the following properties:
  \begin{enumerate}
  \item \label{CCC1} The support of $F$ is the union of disjoint annuli $A_1,\ldots,A_k$ which are foliated by closed leaves;
  \item \label{CCC2} each annulus $A_i$ is a regular neighborhood of a closed curve $c_i$, where $c_1,\ldots,c_k$ are the components of a  system of curves representing the isotopy class $C$;
  \item \label{CCC3} for $1\leq i\leq k$, the total transverse measure of the annulus $A_i$ is equal to  the $i$-th coordinate of $v$.
  \end{enumerate}

We shall call a foliation on $S$ representing an element of  $\mathcal{MF}$ which is the image of some element of $\mathcal{S}'$ by one of the maps $j_k$ an {\it annular foliation}.

\begin{proposition}\label{prop:MFk} For every $k\geq 1$, the image of $(\mathbb{R}^*)^k\times\mathcal{S}_k$ in $\mathcal{MF}$ is the set $\mathcal{MF}_k$. 
\end{proposition}

\begin{proof} 
Let $F\in \mathcal{MF}$ be a measured foliation class which is in the image of $(\mathbb{R}^*)^k\times\mathcal{S}_k$. We must show that it satisfies Properties (\ref{CCC1}) to (\ref{CCC3})  above. The proof is based on techniques used in \cite{P1} Chapter 1. The idea is as follows. We start by representing $F$ by a system of weights on a train track induced by a union of $k$ disjoint simple closed curves, representing the given element of $(\mathbb{R}^*)^k\times\mathcal{S}_k$. We pinch this system of curves along a system of disjoint arcs having their endpoints on these curves in order to obtain a maximal recurrent train track $\tau$ such that $F$ is in the interior of the associated  polytope  $V_{\tau}$. A pinching operation is represented in Figure \ref{fig:pinching}.  Now by choosing a different system of arcs, we can obtain a maximal recurrent train track  $\sigma$ such that $F$ is in the interior of the polytope  $V_{\sigma}$ associated to $\sigma$, and we can choose the new system of arcs so that $F$ is in the dimension-$k$ stratum of the flag  $\mathrm{Fl}(\psi_{\tau^{-1}\sigma})$ of the coordinate change map. This uses the description of the coordinate changes that is contained in Proposition \ref{prop:change} above. With this, it is easy to see that $F$ is on a singular set that has the required properties and is an element of $\mathcal{MF}_k$.
 
 Conversely, any measured foliation satisfying Properties  (\ref{CCC1}) to (\ref{CCC3})   is in the image of $(\mathbb{R}^*)^k\times\mathcal{S}_k$.  Note that Condition (\ref{CCC3}) ensures that the measured foliation is the class of a collection of weighted curves, that is, all its components are annular.

\end{proof}

 \begin{figure}[!hbp]
 \begin{center}
\scalebox{0.6}{\includegraphics{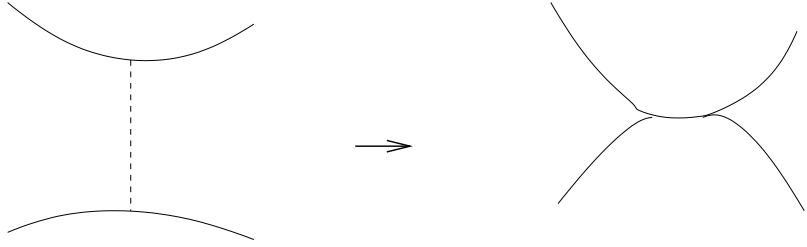}}
\end{center}
\caption{The pinching operation that is used in the proof of Proposition \ref{prop:MFk}.}
\label{fig:pinching}
\end{figure}

\begin{theorem}[\cite{Papadopoulos2008}] \label{th:autPL} Assume that $S$ is a closed surface of genus $\geq 3$.
Then, the natural homomorphism \[\Gamma^*(S)\to \mathrm{Aut}(\mathcal{PMF},\mathcal{P})\] is an isomorphism. In the case where $S$ is a closed surface of genus 2, this homomorphism is surjective, and its kernel is isomorphic to $\mathbb{Z}_2$ and is generated by the hyperelliptic involution.
\end{theorem}

\begin{proof}
Let $f$ be an element of $ \mathrm{Aut}(\mathcal{PMF},\mathcal{P})$. By Proposition \ref{prop:preserves}, $f$ preserves the subset $\mathcal{PMF}_1$ of $\mathcal{PMF}$. By Proposition \ref{prop:MFk}, $\mathcal{PMF}_1$ is the natural image of $\mathbb{R}_+^*\times\mathcal{S}$ in $\mathcal{MF}$, that is, the set of equivalence classes of measured foliations that are representable by foliations whose nonsingular leaves are all closed curves homotopic to a connected simple closed curve. Thus, $\mathcal{PMF}_1$ is in natural one-to-one correspondence with the set $\mathbb{R}_+^*\times\mathcal{S}$ of isotopy classes of weighted essential curves on $S$. 

Since $f$ acts linearly on rays, it acts on the image in $\mathcal{PMF}$ of the set $\mathcal{S}$ of isotopy classes of essential curves,  which is identified with the set of vertices of the curve complex $C(S)$. Therefore, $f$ defines a map  from the vertex set of the curve complex $C(S)$ to itself, and it follows from the fact that $f$ is a homeomorphism that this map is a bijection.

Similarly, by Proposition \ref{prop:preserves}, for each $k=2,\ldots,3g-3$, $f$ preserves the set $\mathcal{PMF}_k$ of $\mathcal{MF}$ which, again by Proposition \ref{prop:MFk}, is naturally identified with the set $\mathcal{S}_k$ of isotopy classes of weighted systems of curves which have $k$ components. Thus, $f$ also induces a map of the set of $(k-1)$-simplices of $C(S)$. Therefore, the bijection induced by $f$ on the vertex set of $C(S)$ can be naturally extended to a simplicial automorphism of $C(S)$. 

By Theorem \ref{th:IKL}, the action of $f$ on $C(S)$ is induced by an element $\gamma$ of the extended mapping class group of $S$. It is clear from the definitions of these actions that the restriction of $f$ and of the extended mapping class $\gamma$  on the image of  $\{1\}\times \mathcal{S}$ (and, even, of $\{1\}\times \mathcal{S}'$) in $\mathcal{MF}$ coincide.

Since the image of $\mathcal{S}$ in $\mathcal{PMF}$ is dense and since the actions of $f$ and $\gamma$ on $\mathcal{PMF}$ are continuous, the two actions coincide on $\mathcal{PMF}$. Thus, each automorphism of $(\mathcal{PMF},\mathcal{P})$ is induced by an extended mapping class. This proves the surjectivity of the homomorphism $\Gamma^*(S)\to \mathrm{Aut}(\mathcal{PMF},\mathcal{P})$. The injectivity  in genus $\geq 3$ and the fact that in genus 2 the kernel is $\mathbb{Z}_2$ follow from the fact that the homomorphism from the extended mapping class group to the automorphism group
 of the curve complex is injective except in genus two and in that case the kernel is $\mathbb{Z}_2$.\end{proof}

The rigidity result, Theorem \ref{th:autPL}, can be seen as a rigidity theorem for some boundary structure of Teichm\"uller space. We mention the following:
     \begin{problem}
   Find a proof of Royden's rigidity theorem on the isometry group of the Teichm\"uller metric by using  Theorem \ref{th:autPL}. 
     \end{problem}
One approach would be to use the Finsler structure of the Teichm\"uller metric, where the unit ball at each point at each point is the space of projective measured foliations. The questions is then to introduce somehow in this picture the piecewise-linearity of the action on projective foliation space.

Another natural problem is the following:
\begin{problem}
   Find rigidity results similar to Theorem \ref{th:autPL}, for the other boundaries of Teichm\"uller space, each boundary equipped with an appropriate structure.
     \end{problem}
     One may think of the Gardiner-Masur boundary \cite{GM1}, of the Poisson boundary \cite{KM1}, of the boundary of the augmented Teichm\"uller space \cite{Ab}, and there are others.  We also refer to the paper \cite{Ohshika-H} for a survey on compactifications of Teichm\"uller space.

   \subsection{Homeomorphisms of the space of unmeasured foliations} \label{s:unmeasured}

In this section, $S=S_{g,n}$ is a surface of genus $g\geq 0$ with $n\geq 0$ punctures and $\mathcal{UMF}=\mathcal{UMF}(S)$ is the quotient of the space $\mathcal{MF}$ of measured foliations on $S$ by the equivalence relation which identifies  two elements whenever they can be represented by topologically equivalent foliations, that is, foliations that are isotopic after forgetting the transverse measure. We call the elements of $\mathcal{UMF}$ \emph{unmeasured foliations} and we call $\mathcal{UMF}$ \emph{unmeasured foliation space}. It is equipped with the quotient topology inherited from $\mathcal{MF}$.

Given a measured foliation $F$, we shall use the same notation $[F]$ for its images in $\mathcal{MF}$ and in $\mathcal{UMF}$, specifying explicitly the space to which $[F]$ belongs whenever this is necessary.

 In this subsection, we present the result saying that the action of the extended mapping class group on $\mathcal{UMF}$ is rigid. Before that, let us make a few comments on the appearance of $\mathcal{UMF}$ in various contexts.

 \medskip
  \begin{enumerate}
\item  
Klarreich \cite{Klarreich} identified the Gromov boundary of
$C(S)$ with the subspace $\mathcal{EF}$ of $\mathcal{UMF}$ consisting of equivalence classes of foliations satisfying the following two properties:
 \begin{enumerate}  
 \item \label{KK1} the foliation is \emph{maximal} in the sense that its intersection number with any simple closed curve on $S$ is nonzero;
  \item \label{KK2}  the foliation is \emph{minimal} in the sense that every leaf is dense in the support.
 \end{enumerate}
We shall see that $\mathcal{UMF}$ is a non-Hausdorff topological space. But unlike  $\mathcal{UMF}$, the subspace $\mathcal{EF}$ (being a Gromov boundary) is Hausdorff. This is compatible with the fact (which we shall see below) that the non-Hausdorffness of $\mathcal{UMF}$ is concentrated at the measured foliations that do not satisfy one of the two above properties. The Gromov boundary of the curve complex was used as an essential ingredient in the proof of Thurston's Ending Lamination Conjecture.  In the theory of end invariants of hyperbolic 3-manifolds, convergence of a sequence of simple closed geodesic towards the ending lamination of a simply degenerate end is interpreted in terms of convergence of the vertices in the curve complex representing these curves towards a point in $\mathcal{EF}$. We refer the reader to \S 4.2 of \cite{Ohshika-H} for a short account on end invariants.

  \item Thurston conjectured that the Bers boundary of Teichm\"uller space, up to quasiconformal equivalence, is homeomorphic to unmeasured lamination space (which is closely related to unmeasured foliation space). We shall see this more precisely below, in the work of Ohshika on the reduced Bers boundary, who proved Thurston's conjecture in a more precise form (see \S \ref{s:reduced}).

\item If we take the quotient of the Teichm\"uller compactifiation by  replacing Thurston's boundary by unmeasured foliation spacen then the action of the mapping class group extends continuously to this new space. This was claimed by Kerckhoff in \cite{Ke}; see also the discussion in Ohshika in \cite{Ohshika-H}.

 \end{enumerate}

  The extended mapping class group $\Gamma^*= \Gamma^*(S)$  of $S$ acts naturally by homeomorphisms on $\mathcal{UMF}$.  

 Let $\mathrm{Homeo}(\mathcal{UMF}(S))$ be the group of homeomorphisms of $\mathcal{UMF}(S)$. We have the following.

\begin{theorem}[Ohshika \cite{Ohshika2012}] \label{th:Ohshika2012} Suppose that $S$ is not a sphere with at most four punctures or a torus with at most two punctures and consider the natural homomorphism 
\[\Gamma^*(S) \to \mathrm{Homeo}(\mathcal{UMF}(S)).\]
Then,
\begin{enumerate}
\item If $S$ is not  the closed surface of genus 2, this homomorphism is an isomorphism. 
\item If  $S$ is the closed surface of genus 2, this homomorphism is surjective and not injective, and its kernel is $\mathbb{Z}_2$.
\end{enumerate}
\end{theorem}

Theorem \ref{th:Ohshika2012} shows that the space $\mathcal{UMF}(S)$, unlike the space $\mathcal{MF}(S)$,  is topologically very inhomogeneous. Its homeomorphism group is countable, therefore this space does not contain any open set which is a manifold of positive dimension. This contrasts with the space $\mathcal{MF}$ of measured foliations, which is a manifold.

The proof of Theorem \ref{th:Ohshika2012} is based on extensions of techniques that were introduced in \cite{Papadopoulos2007}, where the following result was proved:

\begin{proposition}[Papadopoulos \cite{Papadopoulos2007}]\label{prop:Papadopoulos2007} 
   Let $S$ be an orientable surface of finite type which is not a sphere with at most four punctures or a torus with at most two punctures and let  $\mathcal{D}$ be the natural image in $\mathcal{UMF}$ of the set $\mathcal{S}'$ of systems of curves on $S$.   Then, 
 \begin{enumerate} 
 \item \label{PP1} $\mathcal{D}$ is dense in $\mathcal{UMF}$, and it is invariant by the group $\mathrm{Homeo}(\mathcal{UMF})$. Furthermore, if $h$ is any homeomorphism of $\mathcal{UMF}$, then there exists an element $h^*$ of $\Gamma^*(S)$ such that the restriction of the actions of $h$ and $h^*$ on $\mathcal{D}$ coincide.
 
 \item \label{PP2}   In the case where $S$ is not the closed surface of genus 2, if $h_1$ and $h_2$ are distinct elements of $\Gamma^*(S)$, then their induced actions on $\mathcal{D}$ are distinct. In particular, the natural homomorphism $\Gamma^*(S)\to \mathrm{Homeo}(\mathcal{UMF})$ is injective. 
 
 \item \label{PP3}  In the case where $S$ is the closed surface of genus 2,  the kernel of the homomorphism from $\Gamma^*$ to the homeomorphism group of $\mathcal{D}$ is $\mathbb{Z}_2$.
 \end{enumerate}
\end{proposition}

The proofs of Proposition \ref{prop:Papadopoulos2007} and of Theorem \ref{th:Ohshika2012} are based on the fact that the topology of the space $\mathcal{UMF}$ is non-Hausdorff, and that one can measure precisely the non-Hausdorffness of this space. The set $\mathcal{S}$ of homotopy classes of curves is naturally embedded in $\mathcal{UMF}$, and for each $n$, the set of systems of curves having $n$ components can be characterized by one of the levels of this non-Hausdorffness. This shows that this set has a topological significance and is invariant by any homeomorphism of $\mathcal{UMF}$. From this, we get an action of an arbitrary homeomorphism of  $\mathcal{UMF}$ on the curve complex, and we can use Theorem \ref{th:IKL} to conclude that this action is induced by a homeomorphism of the surface.

We shall give the important steps in the proofs of Proposition \ref{prop:Papadopoulos2007} and Theorem \ref{th:Ohshika2012}.

 We make use again of partial foliations (\S \ref{sectionPL}). If $F$ and $G$ are two partial measured foliations on $S$ with disjoint supports, then their union can be naturally considered as a (partial) measured foliation on $S$, which we shall denote by $F+G$.
We shall say that $F$ is a {\it subfoliation} of $F+G$ and that the foliation $F+G$ (or any foliation equivalent to $F+G$) {\it contains} the foliation $F$ (or any foliation equivalent to $F$).

Consider now two distinct equivalence classes of measured foliations $[F_1]$ and $[F_2]\in \mathcal{MF}$, represented by disjoint partial measured foliations $F_1$ and $F_2$. For each positive numbers $t_1\not= t_2$, the elements $t_1F_1 +F_2$ and $t_2F_1 +F_2$ represent distinct elements in $\mathcal{MF}$, but they represent the same elements in $\mathcal{UMF}$. For any sequence $t_n$ of positive real numbers converging to 0, the sequence $[t_nF_1 +F_2]$ in $\mathcal{MF}$ converges to the equivalence class $[F_2]$. On the other hand, the sequence  $[t_nF_1 +F_2]$ is  stationary in $\mathcal{UMF}$. Thus, every neighborhood of $[F_2]$ in $\mathcal{UMF}$ contains $[tF_1 +F_2]$. Note that the converse is not true, that is, there exist neighborhoods of the image of  $tF_1 +F_2$ in  $\mathcal{UMF}$ that do not contain the image of $F_2$ in that space.

 To express this property in a proposition, we first introduce some terminology. If $x$ and $y$ are two points in a (non-Hausdorff) topological space, let us say that \emph{$x$ is separated from $y$} if there exists a neighborhood of $x$ that does not contain $y$.

From the preceding remarks, we have the following:

\begin{proposition} \label{non-h} The space $\mathcal{UMF}$ satisfies the following properties:
\begin{enumerate} 
\item  \label{e11}  it is not Hausdorff;
\item  \label{e12} the relation ``$x$ is separated from $y$" is not symmetric.
 \end{enumerate}
\end{proposition}
Both properties are used in an essential way in this theory. In fact, it is by analyzing precisely the non-Hausdorffness property that Proposition \ref{prop:Papadopoulos2007} is proved. The non-symmetry of points separation property  is used to complete the proof of Theorem \ref{th:Ohshika2012}.

The proof of Proposition \ref{prop:Papadopoulos2007} involves the notions of \emph{adherence}, \emph{adherence set}, \emph{complete adherence set} and {\it adherence number}.  These are purely topological notions  and they were introduced in \cite{Papadopoulos2007} in the setting of spaces of equivalence classes of measured foliations.

Let $X$ be a topological space.

\begin{definition}[Adherence] \label{def:adherence}
Let $x$ and $y$ be two points in $X$.
We say that $x$ is {\it adherent to} $y$ in $X$ if every neighborhood of $x$ intersects every neighborhood of $y$.
\end{definition}

The relation ``$x$ is adherent to $y$" is symmetric in $x$ and $y$.
\begin{definition}[Adherence set]
Let $x$ be a point in $X$.
The {\it adherence set} of $x$ is the set of elements in $X$ which are adherent to $x$.
\end{definition}

\begin{definition}[Complete adherence set]
A subset $Y$ of $X$ is a {\it complete adherence set} in $X$ if for any two elements $x$ and $y$ of $Y$, $x$ is adherent to $y$ in $X$. \end{definition}

\begin{definition} Let $x$ be a point in $X$. The
{\it adherence number} $\mathcal{N}(x)$ of $x$ is the element of $\mathbb{N}\cup\{\infty\}$ defined as
\[\mathcal{N}(x)=\sup\{ \mathrm{Card}(A)\ \vert \ x\in A\text{ and } A \text{ is a complete adherence set in $X$}\}.\]
\end{definition}
 
Now we study these notions in the context of the non-Hausdorff space $\mathcal{UMF}$.

We denote by $\sim$ the equivalence relation between measured foliations that define the elements of $\mathcal{MF}$.
 
\begin{lemma}\label{lem:int} Let $F$ and $G$ be two measured foliations on $S$. The following are equivalent:
\begin{enumerate}
\item\label{prop:int1} $i(F,G)= 0$.
\item\label{prop:int2} $F\sim F'$ and $G\sim G'$, where $F'$ and $G'$ are partial measured foliations on $S$ such that $F'=F_1+F_2$, $G'=G_1+G_2$,   $F_1$ and $G_1$ are equal as topological foliations, and  $F_2$ and $G_2$ have disjoint supports. (Some of the partial foliations $F_1$, $F_2$, $G_1$ and $G_2$ may be empty.)
\item \label{prop:int3}  $[F]$ is in the adherence set of $[G]$ in $\mathcal{UMF}$.
\end{enumerate}

\end{lemma}   

The following proposition is a direct consequence of the equivalence (\ref{prop:int1})$\Leftrightarrow$(\ref{prop:int2}) of Lemma \ref{lem:int}.

\begin{proposition}[Proposition 3.2 of \cite{Papadopoulos2007}]\label{prop:0} Let $F$ be a measured foliation and let $[F]$ be its image in $\mathcal{UMF}$. Then, the adherence set of $F$ is the set of equivalence classes in $\mathcal{UMF}$ of foliations $G$ which are of the form $G_1+G_2$ where $G_1$ is a sum of components of $F$ and where $G_2$ is a partial measured foliation whose support is disjoint from the support of a representative of $F$ by a partial measured foliation.
\end{proposition}

\begin{proposition}[Proposition 3.3 of  \cite{Papadopoulos2007}]\label{prop:max} Let $F$ be a measured foliation on $S$ and let $[F]$ denote the corresponding element of $\mathcal{UMF}$. Then, 
$\mathcal{N}([F])=2^q -1$, where $q$ is the maximum, over all measured foliations $G$ containing $F$, of the number of components of $G$.
\end{proposition}

We shall call a foliation on $S$ representing an element of  $\mathcal{UMF}$ which is the image of some element of $\mathcal{S}'$ by the map $j$ an {\it annular foliation}.

\begin{proposition}[Proposition 3.4 of  \cite{Papadopoulos2007}]\label{prop:nb}
If $x\in\mathcal{UMF}$ is the equivalence class of an annular foliation, then $\mathcal{N}(x)= 2^q-1$, with $q=3g-3$. Furthermore, if
$y\in\mathcal{UMF}$ is not the equivalence class of an annular foliation, 
 then $\mathcal{N}(y)< \mathcal{N}(x)$.
\end{proposition}

We denote by $j:\mathcal{S}' \hookrightarrow \mathcal{UMF}$ the natural inclusion induced from the map that we considered in Section \ref{sectionPL} which associates to each weighted system of curves the corresponding annular measured foliation.

  The following corollaries follow from Proposition \ref{prop:nb} and of the fact that any homeomorphism of $\mathcal{UMF}$ preserves adherence numbers of points.

\begin{corollary}\label{co}
 Any homeomorphism of $\mathcal{UMF}$ preserves the image of $\mathcal{S}'$  in $\mathcal{UMF}$ by the map $j$. 
 \end{corollary} 
 
\begin{corollary}\label{cor:not}
If $g\not= g'$, then the two spaces $\mathcal{UMF}(S_{g})$ and $\mathcal{UMF}(S_{g'})$ are not homeomorphic.
\end{corollary}
 
As in \S \ref{sectionPL}, for every $k\geq 1$, we denote by of $\mathcal{S}_k$ the subset of $\mathcal{S}'$ consisting of homotopy classes that are representable by pairwise disjoint and non-homotopic $k$ simple closed curves.  

\begin{proposition}[Proposition 3.7 of \cite{Papadopoulos2007}]\label{prop:Sk}
For any $k\geq 1$, any homeomorphism of $\mathcal{UMF}$ preserves the image of $\mathcal{S}_k$ by the map $j:\mathcal{S}'\to\mathcal{UMF}$.
\end{proposition}
\begin{proof}
Let $f$ be a homeomorphism of $\mathcal{UMF}$. By Corollary \ref{co}, $f$ preserves the subset $j(\mathcal{S}')$ of  $\mathcal{UMF}$. Now $j(\mathcal{S}')$ is the disjoint union of the spaces $j(\mathcal{S}_k)$ with  $k=1,\ldots 3g-3$. Thus, it suffices to prove that if $m\not=n$ and if $[F]\in j(\mathcal{S}_m)$ and $[G]\in j(\mathcal{S}_n)$, we have  $f([F])\not=[G]$. 
By Proposition \ref{prop:0}, the adherence set $\mathcal{A}([F])$ of $[F]$ (respectively $\mathcal{A}([G])$ of $[G]$) in $\mathcal{UMF}$ is homeomorphic to a finite union of spaces which are all homeomorphic to a space $\mathcal{UMF}(S')$ (respectively  $\mathcal{UMF}(S'')$) where $S'$ and  $S''$ are (not necessarily connected) subsurfaces of $S$ 
which are the complement of the support of partial  foliations $F$  and $G$ representing $[F]$ and $[G]$ respectively. Since the number of components of $[F]$ and $[G]$ are distinct, then the genera of $S'$ and $S''$ are different.  By Corollary \ref{cor:not},  $\mathcal{A}([F])$ is not homeomorphic to $\mathcal{A}([G])$. Thus, $f$ cannot send $[F]$ to $[G]$, which completes the proof of the proposition. 
\end{proof}

\begin{proof}[Proof of Proposition \ref{prop:Papadopoulos2007}.]
 The fact that $\mathcal{D}$ is dense in $\mathcal{UMF}$ and is invariant by the group $\mathrm{Homeo}(\mathcal{UMF})$ follows from the corresponding property in $\mathcal{MF}$. We prove the rest of Part (\ref{PP1}) of this proposition. 

Let $h$ be an arbitrary homeomorphism of $\mathcal{UMF}$. Since $h$ preserves the set $j(\mathcal{S}_1)=j(\mathcal{S})$, it induces a map from the vertex set $j(\mathcal{S})$ of $C(S)$ to itself. Since $h$ is a homeomorphism, this induced map is bijective. Furthermore, since, for each $k\geq 2$, $h$ preserves the set $j(\mathcal{S}_k)$ in $\mathcal{UMF}$, this action on the vertex set of $C(S)$ extends naturally to a simplicial automorphism $h'$ of $C(S)$.  If $S$  is not the closed surface of genus two, it follows from   Theorem \ref{th:IKL} that the automorphism $h'$ is induced by an element $h''$ of $\Gamma^*(S)$. 
 The element $h''$ acts on $\mathcal{UMF}$, and it is easy to see from the construction that the action induced on $\mathcal{D}$ by this map is the same as that of $h$. This completes the proof Part (\ref{PP1}). 
 
The proofs of Parts (\ref{PP2}) and  (\ref{PP3}) follow from the fact that if $S$ is not the closed surface of genus two, the natural homomorphism from the extended mapping class group to the curve complex is injective (see Theorem \ref{th:IKL}) and that in the case where $S$ is the closed surface of genus two, the kernel of that homomorphism is $\mathbb{Z}_2$.  
 
 \end{proof}

In the rest of this section, we prove Theorem \ref{th:Ohshika2012}. The proof we give is the one of Ohshika \cite{Ohshika2012} who introduced the following asymmetric version of Definition \ref{def:adherence}:

\begin{definition}[Unilateral adherence] \label{def:unilateral-adherence} If $x$ and $y$ are two points in a topological space $X$, we say that $x$ is \emph{unilaterally adherent} to $y$ if $x\not=y$ and every neighborhood of $x$ contains $y$.
\end{definition}
We already saw that in $\mathcal{UMF}$ this relation is not symmetric.
In fact, we have the following:
\begin{lemma}[Lemma 1 of \cite{Ohshika2012}] \label{lem:uni} Let $x$ and $y$ be two distinct elements in $\mathcal{UMF}$. Then, the following are equivalent:
\begin{enumerate}
\item \label{lem-uni1} $x$ and $y$ can be represented by partial measured foliations $F$ and $G$ respectively in such a way that $G$ is a subfoliation of $F$ and is distinct from $F$;
\item \label{lem-uni2}  $x$ is unilaterally adherent to $y$.
\end{enumerate}
\end{lemma}
 
From Lemma \ref{lem:uni}, unilateral adherence in $\mathcal{UMF}$ can be translated into a relation of containment between foliations. This lemma also expresses the notion of unilateral adherence among foliations in terms of a set-theoretic property, which can be translated into a topological property in $\mathcal{UMF}$. This topological property is then used in the proof of Theorem  \ref{th:Ohshika2012} as an invariant property by homeomorphisms of $\mathcal{UMF}$.

\begin{lemma}[Lemma 2 of \cite{Ohshika2012}] \label{lem:Ohsh-curves} Let $x$ be an element of $\mathcal{UMF}$ represented by a partial measured foliation $F$ and let $F_n$ be a sequence of partial foliations representing elements of $\mathcal{UMF}$ and converging geometrically to $F$, in the sense that $F_n$ has the same support as $F$ for every $n$ and  the leaves of $F_n$ converge as unoriented line fields to the leaves of $F$. (Note that this implies that the image $[F_n]$ in $\mathcal{UMF}$ converges to $x$.) Let $y$ be an element of $\mathcal{UMF}$ such that $x$ is not unilaterally adherent to $y$. Then the sequence $[F_n]$ in $\mathcal{UMF}$ does not converge to $y$.
\end{lemma}
 The following definitions were also introduced by Ohshika.

\begin{definition}[Adherence tower]
Given an element $x$ in $\mathcal{UMF}$, an \emph{adherence tower} for $x$ is a sequence $x=x_0,x_1,\ldots,x_m$ of elements in $\mathcal{UMF}$ such that for every $0\leq i\leq k-1$, $x_i$ is unilaterally adherent to $x_{i+1}$.
\end{definition}

\begin{definition}[Adherence height]
Given an element $x$ in $\mathcal{UMF}$, the \emph{adherence height} of $x$, denoted by $\mathrm{a.h.}(x)$, is the  the maximal integer $m$ such that $x=x_0,x_1,\ldots,x_m$ is an
adherence tower  for $x$.
\end{definition}

The following lemma expresses the fact that adherence height is a topological property, and this is used in the proof of Theorem \ref{th:Ohshika2012} that we give below.

\begin{lemma}\label{lem:ad-h}
For any homeomorphism $f$ of $\mathcal{UMF}$ and for any $x$ in this space, we have  $\mathrm{a.h.}(f(x)) = \mathrm{a.h.}(x)$.
\end{lemma}
 
The following lemma will also be used in the proof of Theorem \ref{th:Ohshika2012}. Its proof follows from Lemma \ref{lem:uni}.
\begin{lemma}[Lemma 3 of \cite{Ohshika2012}] \label{lem:ad}
Let $x$ and $y$ be two elements of $\mathcal{UMF}$. If $x$ is unilaterally adherent to $y$ then $\mathrm{a.h.}(y) < \mathrm{a.h.}(x)$.
\end{lemma}

\begin{proof}[Proof of Theorem \ref{th:Ohshika2012}] Assume that $S$ is not the closed surface of genus 2 and let us show that the natural homomorphism $\Gamma^*(S) \to \mathrm{Homeo}(\mathcal{UMF}(S))$ is surjective. Let $h$ be a homeomorphism of $\mathcal{UMF}$. By Proposition \ref{prop:Papadopoulos2007}, there exists an element $h^*$ of $\Gamma^*(S)$ that has the same action on the natural image $\mathcal{D}$ in $\mathcal{UMF}$ of the set $\mathcal{S}'$ of systems of curves on $S$. We need to show now that the actions of $h$ and $h^*$ are the same on $\mathcal{UMF}$. 

Let $x$ be an element of $\mathcal{UMF}$ which is not in $\mathcal{D}$, let $F$ be any measured foliation representing $x$ and consider a sequence $F_n$ of (partial) measured foliations converging geometrically to $F$ in the sense of Lemma \ref{lem:Ohsh-curves}. 

Since $h$ and $h^*$ induce the same action on $\mathcal{D}$, the sequence $(h([C_i]))= h^*([C_i])$ converges in $\mathcal{UMF}$ to both $h(x)$ and $h^*(x)$. By Lemma \ref{lem:Ohsh-curves}, either $h^*(x)=h(x)$ or $h(x)$ is unilaterally adherent to $h^*(x)$. 
In the latter case, by Lemma \ref{lem:ad}, we have $\mathrm{a.h.}(h^*(x)) < \mathrm{a.h.}(h(x))$. By Lemma \ref{lem:ad-h}, we have  $\mathrm{a.h.}(h^*(x)) = \mathrm{a.h.}(x)= \mathrm{a.h.}(h(x))$, which is a contradiction. Thus, $h^*(x)=h(x)$. This proves that the homomorphism  $\Gamma^*(S) \to \mathrm{Homeo}(\mathcal{UMF}(S))$ is surjective.

The injectivity and the case  where $S$ is a closed surface of genus 2 are treated as in Proposition \ref{prop:Papadopoulos2007}.
\end{proof}
\subsection{The action on the reduced Bers boundary}\label{s:reduced}

In 1970, Bers introduced a compactification of Teichm\"uller space which is known today as the \emph{Bers compactification}  \cite{Bers1970}. It is obtained via an embedding of Teichm\"uller space into a space of holomorphic quadratic differentials of a given Riemann surface (the basepoint of Teichm\"uller spac) using the so-called simultaneous uniformization and the Schwarzian derivative. This embedding was shown by Bers to be complex analytic, and in fact it can be used as one way (among others) of equipping Teichm\"uller space with a complex structure.   
This compactification provides Teichm\"uller space with a boundary called the \emph{Bers boundary}. The definition of the Bers compactification uses the choice of the basepoint of Teichm\"uller space, and by a result of Kerckhoff and Thurston, the Bers boundary indeed depends on the choice of the basepoint. Furthermore, Kerckhoff and Thurston showed that the mapping class group does not extend continuously to an action on a reduced Bers boundary. 
  
 Ohshika examined a quotient of the Bers boundary -- the \emph{reduced Bers boundary} -- by an equivalence relation. He showed that the resulting boundary is independent of the choice of the basepoint, and that the action of the mapping class group on Teichm\"uller space extends continuously to the reduced Bers boundary.  He also obtained a rigidity result for the action of the mapping class group on the reduced Bers boundary (Theorem \ref{th:Ohshika1} below). 

The reduced Bers boundary is a subset of unmeasured foliation space, but (as Ohshika proved) its topology is different from that induced from the quotient topology considered in \S \ref{s:unmeasured}. There is a relation between the reduced Bers boundary and 3-manifold topology, in particular the Ending Lamination Conjecture. An ending lamination is itself an unmeasured lamination (which corresponds to an unmeasured foliation). We just recall here that the deformation space of geometrically finite hyperbolic structures on the interior of a compact hyperbolic 3-manifold $M$ is parametrized by the Teichm\"uller space $\mathcal{T}$ of its boundary (which is reduced to a point if the manifold has no boundary) and that in the case of geometrically infinite hyperbolic structures, an ending lamination is an invariant of such a structure on $M$. Provided the hyperbolic manifold has no parabolics, the ending lamination of a geometrically infinite end completely describes the end. We refer the reader to the paper \cite{Ohshika2011} and to the survey \cite{Ohshika-H} by Ohshika.

Ohshika proved in \cite{Ohshika2011}  that up to a few exceptional cases, any homeomorphism of a reduced Bers boundary is induced by a unique mapping class.
The proof is based again on an analysis of the non-Hausdorffness of the topology of the reduced Bers boundary, and on the introduction of a hierarchy in this non-Hausdorffness,  measured by adherence height. Thus, the proof is similar in spirit to the proofs of Theorems \ref{th:Ohshika2012} and \ref{main theorem}. In fact, adherence height is used to show that every homeomorphism of a Bers slice preserves the set of the so-called geometrically finite b-groups and the number of cusps, and from there, the action of a homeomorphism is shown to induce an action on the curve complex of the surface. The uniqueness part in the result is however more delicate and it uses work of Ohshika on Kleinian groups from his paper \cite{Ohshika-Div}.  

Denoting the reduced Bers boundary of the Teichm\"uller space of a surface $S$ by $\partial^{RB}(S)$, the result of Ohshika is as follows:

\begin{theorem}[Ohshika \cite{Ohshika2011} Theorem 4.1]\label{th:Ohshika1}
Let $S$ be a surface of finite type which is not the sphere with $\leq 4$ punctures or the torus with 0 or 1 hole. 
For any homeomrophism $f:\partial^{RB}(S)\to \partial^{RB}(S)$, there is a homeomorphism $h:S\to S$ 
whose induced action on $\partial^{RB}(S)$ is $f$. Furthermore, if $S$ is not the closed surface of genus 2, then 
any two diffeomorphisms of $S$ that induce the same action on $\partial^{RB}(S)$ are isotopic.  

\end{theorem}

\begin{corollary}[Ohshika \cite{Ohshika2011} Corollary 4.9]\label{th:Ohshika2}
If $S_1$ and $S_2$ are two surfaces satisfying $\mathrm{dim}(\mathcal{T}(S_1))\not= \mathrm{dim}(\mathcal{T}(S_2))$, then the corresponding reduced boundaries $ \partial^{RB}(S_1)$ and $ \partial^{RB}(S_2)$ are not homeomorphic.
\end{corollary}

\subsection{Homeomorphisms of the space of geodesic laminations}

In this section, we study the space  $\mathcal{GL}$  of geodesic laminations on a hyperbolic surface. This space was introduced in  Chapter 8 of Thurston's Princeton notes \cite{Thurston1979} which concerns mainly the theory of Kleinian groups. Hyperbolic surfaces equipped with their laminations arise in this theory through maps from surfaces into hyperbolic 3-manifolds and they give information on the structure of ends of these manifolds. This study is also at the origin of Thurston's Ending Lamination Conjecture which we already mentioned.

We shall consider the space $\mathcal{GL}$ equipped with the mapping class group action by homeomorphisms, and we shall review a rigidity theorem for this action (Theorem \ref{main theorem} below) due to Charitos, Papadoperakis and Papadopoulos. Before that, we shall recall the definition and some properties of the space $\mathcal{GL}$.

  $S=S_{g,n}$ is again a surface of finite type, of negative Euler characteristic, of genus $g\geq 0$ with $n\geq 0$ punctures. All the hyperbolic structures that we consider on $S$ which are complete and of finite area.  
We shall denote by the same letter $S$ the surface equipped with a hyperbolic structure.
 
A geodesic lamination on $S$ is said to be \emph{minimal} if it contains no proper sublamination. This condition is equivalent to the fact that every leaf of $\Lambda$ is dense in the support of $\Lambda$.

A geodesic lamination on $S$ either contains a finite number of leaves (and in this case it is said to be \emph{finite}), or it has uncountably many leaves.

A \emph{component} of a geodesic lamination $\Lambda$ is a minimal sublamination of $\Lambda$. 
For any geodesic lamination
$\Lambda$ on $S$,  the surface $S\setminus \Lambda$
has  finitely many connected components and the
completion of each connected component with respect to the metric induced by the metric of $S$ is a complete hyperbolic surface of finite area with geodesic boundary. This implies that each lamination has finitely many components.

A geodesic lamination is said to be \emph{maximal} if it is not a proper lamination of another lamination. Equivalently, each complementary component of the lamination is a hyperbolic ideal triangle. 

Any geodesic lamination can be rendered maximal by inserting into its complementary regions a  finite number of geodesics such that this component becomes a union of hyperbolic ideal triangles. The geodesics added can always be chosen to be open geodesics. One may also add closed geodesics, but this does not suffice for completing the lamination.

We equip the set $\mathcal{GL}(S)$ with the \emph{Thurston topology} (see \cite{CEG}, Def. I.4.1.10). We recall the definition.

\begin{definition}[Thurston topology on $\mathcal{GL}(S)$] The \emph{Thurston topology} on $\mathcal{GL}(S)$
is the topology with sub-basis the sets of the form
$$\mathcal{U}_{V}=\{\Lambda \in\mathcal{GL}(S):\Lambda\cap V\neq\emptyset\}$$
where $V$ varies over the collection of open subsets of $S$.  
\end{definition}
 
The topology
$\mathcal{T}$ does not satisfy the first axiom of separability. Indeed, consider a geodesic lamination $\Lambda$ which contains a strict sublamination $\Lambda_1\subsetneqq\Lambda$; then every open set for $\mathcal{T}$ containing $\Lambda_1$ contains $\Lambda$. In particular
the topology $\mathcal{T}$ is not Hausdorff.

 The original reference for the
topology $\mathcal{T}$ is Thurston's Notes \cite{Thurston1979}, Section 8.10, where it is referred to as the \textit{geometric topology}.  There is a natural map  from the space of equivalence classes of projective geodesic laminations with measures of full support (which is in natural one-to-one correspondence with the space $\mathcal{PMF}$ of projective measured foliations) to the space $ \mathcal{GL}$. This map $\mathcal{PML}\to \mathcal{GL}$ is continuous (Proposition 8.10.3 of \cite{Thurston1979}). Thurston uses this map to prove the following:

\begin{proposition}[Proposition 8.10.5 of \cite{Thurston1979}] \label{p:compact}  The space  $\mathcal{GL}(S)$, equipped with Thurston's topology, is compact.
\end{proposition}

Let us also note the following, which is used in the proof of Theorem \ref{main theorem}.
\begin{proposition}[Proposition 8.10.7 of  \cite{Thurston1979}] We have the following:
\begin{enumerate}
\item Geodesic laminations with finitely many leaves are dense in $\mathcal{GL}$.
\item Let $\Lambda\in \mathcal{GL}$ be a geodesic lamination with only finitely many leaves. Then,   each end of a non-compact leaf of $\Lambda$ either spirals around some closed geodesic or it converges to a cusp.
\end{enumerate}
\end{proposition}

  Any homeomorphism $h:S\rightarrow S$ induces a push-forward map
$h_{\ast}:\mathcal{GL}(S)\rightarrow\mathcal{GL}(S)$ which is a homeomorphism
for the topology $\mathcal{T}.$ Let $\mathrm{Homeo}(\mathcal{GL}(S))$ be the space $\mathcal{GL}(S)$ equipped with the Thurston topology.
In \cite{ChPP}, the following is proved:

\begin{theorem}[Charitos-Papadoperakis-Papadopoulos  \cite{ChPP}] \label{main theorem}
Assume that the surface $S$ is not a sphere with at most four punctures or a torus with at most two punctures.
Then the natural homomorphism 
\[\Gamma^*(S)\to \mathrm{Homeo}(\mathcal{GL}(S))
\]
 is an isomorphism. 

\end{theorem}

The proof of Theorem \ref{main theorem} involves the construction, from a homeomorphism $f:  \mathcal{GL}(S)\to  \mathcal{GL}(S)$, of an automorphism of the complex of curves of $S$. One of the main steps in the proof given in \cite{ChPP} is to show that $f$ preserves the set of finite laminations, that is, geodesic laminations which  are finite unions of geodesics. The main idea in the proof is to translate some properties of the Thurston topology (in particular, non-Hausdorffness features) in terms of inclusions between laminations, and in then into set-theoretic properties (that is, topological properties of the space $\mathcal{GL}(S)$).

The following notion of a chain of sublaminations and its length is useful in the proof of the main theorem. It is a measure of the non-Hausdorffness of the space $\mathcal{GL}(S)$, and it should be compared to the notion of adherence tower and adherence height which is used in the proof of Theorem \ref{th:Ohshika1}.

\begin{definition}
Let $\Lambda\in\mathcal{GL}(S).$ A \emph{chain of sublaminations} of $\Lambda$ is a
finite sequence $(\Lambda_{i}),$ $i=0,1,..,n$ of sublaminations of $\Lambda$
such that $\emptyset\neq\Lambda_{n}\subsetneqq\Lambda_{n-1}\subsetneqq
...\subsetneqq\Lambda_{1}\subsetneqq\Lambda_{0}=\Lambda.$ We denote such a chain by
$\mathcal{C}_{\Lambda}.$ The integer $n$ is called the \emph{length} of
$\mathcal{C}_{\Lambda}$ and is denoted by $l(\mathcal{C}_{\Lambda}).$
A chain of sublaminations $\mathcal{C}_{\Lambda}$ is said to be \emph{maximal} if
its length is maximal among all chains of sublaminations of $\Lambda$.  The length of a maximal chain of sublaminations $\mathcal{C}_{\Lambda}$ of
$\Lambda$ depends only on $\Lambda$.
\end{definition}
 
 The number $l(\mathcal{C}
_{\Lambda})$ will be referred to as the \emph{length} of $\Lambda$ and will be denoted by
$\mathrm{length}(\Lambda).$

\begin{proposition}[Lemma 3.7 of \cite{ChPP}]
\label{length} Let $f$ be a homeomorphism of $\mathcal{GL}(S)$ and let $\Lambda_{n}\subsetneqq\Lambda_{n-1}\subsetneqq
...\subsetneqq\Lambda_{1}\subsetneqq\Lambda_{0}=\Lambda$ be a maximal chain of
sublaminations of $\Lambda.$ Then $f(\Lambda_{n})\subsetneqq f(\Lambda
_{n-1})\subsetneqq...\subsetneqq f(\Lambda_{1})\subsetneqq f(\Lambda
_{0})=f(\Lambda)$ is also a maximal chain of sublaminations of $f(\Lambda)$ and
$\mathrm{length}(\Lambda_{k})=\mathrm{length}(f(\Lambda_{k}))$ for each $k=0,1,..,n.$
\end{proposition}

A \emph{generalized pair of pants} is a hyperbolic surface which is homeomorphic to a sphere with three holes, a hole being either a geodesic boundary component or a cusp. 
 
The following proposition is also used in the proof of Theorem \ref{main theorem}.

\begin{proposition}[ Proposition 3.8 of \cite{ChPP}]
\label{finite to finite}
Let $f$ be a homeomorphism of $\mathcal{GL}(S)$ with
respect to the Thurston topology. Then:
\begin{enumerate}
\item \label{item:1}  The homeomorphism $f$ sends any maximal finite lamination which contains a collection of curves that decompose $S$ into generalized pair of pants to a maximal finite lamination that has the same property.
\item \label{item:2} The homeomorphism $f$ sends any laminations whose leaves are all closed to a lamination having the same property. Furthermore if such a
  lamination $\Lambda$ has $k$ components then $f(\Lambda)$ has also $k$ components.
  \item \label{item:3}  The homeomorphism $f$ sends finite laminations to finite laminations.
  \end{enumerate}
\end{proposition}

For the proof  of Theorem \ref{main theorem}, one starts  with is a homeomorphism  $f:\mathcal{GL}(S)\rightarrow\mathcal{GL}
(S)\mathcal{\ }$ for the topology $\mathcal{T}$. We must show that there is
a homeomorphism $h:S\rightarrow S$ such that $h_{\ast}=f.$
There is a natural identification between
curve complex $\mathcal{C}(S)$ and the subset $\mathcal{CGL}(S)$ of $\mathcal{GL}(S)$ consisting of geodesic
laminations whose leaves are simple closed geodesics. 
From Proposition \ref{finite to finite}, $f$ induces an automorphism on $\mathcal{C} 
(S)$ and under the hypothesis of Theorem \ref{main theorem}, we obtain from the theorem of Ivanov-Korkmaz-Luo (Theorem \ref{th:IKL})
a homeomorphism $h:S\rightarrow S$ such that $h_{\ast}=f$ on $\mathcal{CGL} 
(S).$  Then we prove 
that $h_{\ast}=f$ on $\mathcal{FGL}(S).$ The details are contained in \cite{ChPP}.
 
The elements of the set $\mathcal{GL}(S)$ can be regarded as compact subspaces of the metric space $S$, and  $\mathcal{GL}(S)$ is then equipped with the Hausdorff metric, induced from the Hausdorff metric on compact subsets of $S$, defined as follows.

For $\epsilon >0$ and for any subset $X$ of $S$, let $N_{\varepsilon}(X)\subset S$ denote the set of points in $S$ which are at distance $\leq \epsilon$ from a point in $X$.

\begin{definition}[Hausdorff distance]
Let $X$ and $X^{\prime}$ be two compact subsets of $S$. The \emph{Hausdorff
distance} between $X$ and $X^{\prime}$  is defined by
$$d_{H}(X,X^{\prime})=\inf\{\varepsilon>0:X\subset N_{\varepsilon} 
(X^{\prime}) \hbox{ and } X^{\prime}\subset N_{\varepsilon}(X\}.$$
\end{definition}

There is no analogue of Theorem \ref{main theorem} for the Hausdorff topology, since simple closed geodesics are all isolated points in  $\mathcal{GL}(S)$ for that topology, and therefore we can find homeomorphisms of  $\mathcal{GL}(S)$ which send curves that all null-homologous to ones that are not null-homologous. Obviously, such homeomorphisms are not induced by surface homemorphisms. We propose the following:

\begin{problem}
Find a weak rigidity result for the action of the extended mapping class group on the space $\mathcal{GL}(S)$ equipped with the Hausdorff topology, that takes into account the existence of these isolated points.
\end{problem}

 \section{Some perspectives}\label{s:perspectives}
   \subsection{Mapping class groups of surfaces of infinite type} 
   In this subsection, we review the theory of mapping class group actions on Teichm\"uller spaces of surfaces of infinite type, and me mention some questions that appear naturally in that theory.

We consider orientable connected surfaces of infinite type, that is, surfaces of infinite genus and/or infinite number of punctures or of boundary components. We recall that such a surface is not topologically classified by its genus (finite or infinite) or its number of punctures or boundary components  (finite or infinite).
In fact, the classification of surfaces of infinite type is much more involved than the one of surfaces of finite type, see e.g. \cite{Richards}. It is also well known that the theory of mapping class groups and their actions have several ramifications in this setting. In fact, there are several groups which may be called ``mapping class group", depending on the context, which are quite different and which are natural generalizations of the case of surfaces of finite type. We will see below some examples.

One difference between the theories for finite and for infinite type surfaces is related to the fact that in the latter case there are several kinds of Teichm\"uller spaces that are associated to the same topological surface, and that these spaces naturally give rise to different mapping class groups acting on them. Even in the case of a fixed topological type of surface, the spaces and the groups depend on the choice of a base structure on the surface, and also on the point of view -- conformal or hyperbolic, and there are also other points of view. In particular, there is a so-called quasiconformal theory and  a hyperbolic theory which are distinct, and in the latter case there is a bi-Lipschitz theory, a length-spectrum theory, and there are other theories that need to be developed. In each case, there is an action of an appropriate mapping class group and a priori the actions are different, at least because the underlying spaces and groups are different.  We shall recall the basic elements of part of the theory which is developed in the papers \cite{ALPS1}, \cite{ALPS2}, \cite{ALPS3}, \cite{ALPS4}, \cite{LLPP}, \cite{Saric1} and \cite{Saric2}. 

For the study Teichm\"uller spaces of surfaces of infinite type and the actions of mapping class groups on them, we also refer to the survey papers \cite{Matsuzaki-Handbook}, \cite{Fu} and \cite{Li-Z}. The first two surveys concern the dynamics of the actions, whereas in the last survey, the author considers geometrical properties such as the non-uniqueness of the geodesics connecting points and the existence of geodesics of arbitrarily length in infinite-dimensional Teichm\"uller spaces.

The rigidity results mentioned in \S \ref{s:rigidity} that concern mapping class group actions on Teichm\"uller spaces of surfaces of finite type hold in the setting of the quasiconformal theory of Teichm\"uller space. In the other settings there are up to now  no equivalent results.

In the rest of this section, $X$ is an oriented surface of topologically infinite type. (We now use the letter $X$ for a topological surface, and the letter $S$ will denote a surface equipped with a specific structure). A first group which might be called the ``mapping class group" of $X$, is defined as a generalization of the case of surfaces of finite type. One starts with the group 
 $\mathrm{Homeo}^*(X)$ (respectively, $\mathrm{Homeo}(X)$) of homeo\-morphisms  (respectively, orientation preserving homeomorphisms) of $X$, equipped with the topology of uniform convergence on compact sets, and with $\mathrm{Homeo}_0^*(X)$ (respectively, $\mathrm{Homeo}(X)$), the identity connected component of that group. As in the case of surfaces of finite type, we set $$\Gamma^*(X)= \mathrm{Homeo}^*(X)/\mathrm{Homeo}_0^*(X)$$
and $$\Gamma(X)= \mathrm{Homeo}(X)/\mathrm{Homeo}_0^*(X).$$ 

It turns out that there are more useful groups which are naturally associated to $X$, and they are subgroups $\Gamma^*(X)$ and $\Gamma(X)$. We shall review them now. We start with the spaces on which these groups act, and we start with the quasi-conformal Teichm\"uller theory.

The definition we give is identical to the definition for surfaces of finite type, but we recall it here because the definitions of the other Teichm\"uller spaces (Definitions \ref{def:bL-i} and \ref{def:ls-i} below) will be modelled on it. In particular, the choice of the basepoint plays now an important role.

 We consider conformal structures on $X$ which have the property that the neighborhood of each topological puncture is conformally a punctured disc. 
The homotopies that we consider preserve the punctures and preserve setwise the boundary components.

 \begin{definition}\label{d:qc-i} 
        Consider a Riemann surface structure $S_0$ on $S$. The {\it  quasiconformal Teichm\"uller space},  $\mathcal{T}_{qc}(S_0)$ the set of equivalence classes $[f,S]$ of pairs $(f,S)$, where $S$ is a Riemann surface homeomorphic to $X$ and where $f:S_0\to S$ is a quasiconformal homeomorphism (whose homotopy class is called the \emph{marking} of $S$) and where two pairs $(f,S)$ and $(f',S')$ are  equivalent (more precisely, they are said to be {\it conformally equivalent})  if there exists a conformal homeomorphism $f'':S\to S'$ homotopic to $f'\circ  f^{-1}$. 
        \end{definition}
        
        The equivalence class of the marked Riemann surface $(\mathrm{Id},S_0)$ is the {\it basepoint} of $\mathcal{T}_{qc}(S_0)$.

      Given two elements $[f,S]$ and $[f',S']$ of $\mathcal{T}_{qc}(S_0)$ represented by marked conformal surfaces $(f,S)$ and $(f',S')$,  their {\it Teichm\"uller distance}  is defined as
        \begin{equation}\label{eq:qc} d_{qc}([f,S],[f',S'])=\frac{1}{2}\log \inf \{K(f'')\}
        \end{equation}
        where the infimum is taken over the set of maximal quasiconformal dilatations $K(f'')$ of quasiconformal homeomorphisms $f'':S\to S'$ homotopic to $f'\circ  f^{-1}$.

To the quasiconformal Teichm\"uller space is associated the following group:

\begin{definition}  The {\it quasiconformal (extended) mapping class group} of $S$, denoted by $\Gamma_{qc}(S)$ (respectively $\Gamma^*_{qc}(S)$) is the subgroup of $\Gamma(X)$  (respectively $\Gamma^*(X)$) consisting of equivalence classes of homeomorphisms that are quasiconformal with respect the base structure $S_0$.
 \end{definition}
 
 The group $\Gamma_{qc}(S_0)$ is a normal subgroup of  index 2 of $\Gamma^*_{qc}(S_0)$.

There is a natural action of $\Gamma^*_{qc}(S_0)$ on the Teichm\"uller space $\mathcal{T}_{qc}(S_0)$, and this action preserves the metric $d_{qc}$.

Unlike the case of a surface of finite type, the action of the quasiconformal mapping class group of a surface of infinite type on the quasiconformal Teichm\"uller space is not properly discontinuous; there are limit sets and domains of discontinuity, and there is a well-developed dynamical theory of the mapping class group action; see the survey \cite{Matsuzaki-Handbook}. This is a generalization of Royden's Theorem which we mentioned in \S \ref{s:rigidity}.
Such a study has not yet been carried out for the other Teichm\"uller spaces that we mention below.

A result whose most general version is due to Markovic \cite{Markovic2003} asserts that any isometry of the space $\mathcal{T}_{qc}(S_0)$ is induced by an element of  $\Gamma^*_{qc}(S_0)$. Here too, there is no analogous results for the other Teichm\"uller spaces.

We now consider other Teichm\"uller spaces.

        We say that a homeomorphism $f:(X_1,d_1)\to (X_2,d_2)$ between metric spaces is  {\it bi-Lipschitz} if there exists a real number $K\geq 1$ satisfying
     \[\label{eq:bl}\frac{1}{K} d_1(x,y)\leq d_2 (f(x),f(y))\leq Kd_1(x,y).
     \]
        The real number $K$ in the above double inequality is called a {\it bi-Lipschitz constant of $f$}.
        Two metric spaces are said to be {\it bi-Lipschitz equivalent} if there exists a bi-Lipschitz homeomorphism between them.

We shall consider hyperbolic structures on $X$, and, as before, they will all be complete with geodesic boundary.
We note however that the completeness issue is delicate in the case of surfaces of infinite type. For instance, the fact that for a surface to be complete it does not suffice (as in the case of surfaces of finite type)  that neighborhoods of all punctures are cusps. For this, and for the other details of the theory,  the reader can  refer to the references that we gave above.

The next definition is modelled on the quasiconformal one (Definition \ref{d:qc-i}).

Let $H_0$ be a hyperbolic metric on the surface $X$. 

        \begin{definition}\label{def:bL-i}  The {\it bi-Lipschitz Teichm\"uller space} of $H_0$, denoted by $\mathcal{T}_{bL}(H_0)$, is the set of equivalence classes $[f,H]$ of pairs $(f,H)$ where $H$ is a hyperbolic metric on a surface homeomorphic to $X$, where $f:H_0\to H$ (the {\it marking} of $H$)  is a bi-Lipschitz homeomorphism, and where two such pairs  $(f,H)$ and $(f',H')$ are equivalent if there exists an isometry $f'':H\to H'$ homotopic to $f'\circ  f^{-1}$.
         \end{definition}

         The topology of $\mathcal{T}_{bL}(H_0)$ is the one induced by the {\it bi-Lipschitz} metric $d_{bL}$, defined by 
                 \[d_{bL}([f,H],[f',H'])=\frac{1}{2}\log \inf \{K\}\]
                  where the infimum is taken over
        all bi-Lipschitz constants $K$ of
        bi-Lipschitz  homeomorphisms $f'':H\to H'$  homotopic to $f'\circ  f^{-1}$, where $(f,H)$ and $(f',H')$ are as usual two hyperbolic surfaces representing the two points
         $[f,H]$ and $[f',H']$   in
 $\mathcal{T}_{bL}(H_0)$. That this defines indeed a metric on
$\mathcal{T}_{bL}(H_0)$ uses the fact that if we take a sequence of bi-Lipschitz homeomorphisms homotopic to $f'\circ f^{-1}$ whose bi-Lipschitz constants tend to 1, then there exists an isometry homotopic to  $f'\circ f^{-1}$. 
                          A result in \cite{Thurston1997} (p. 268) says that for every hyperbolic structure $H$, we have a set-theoretic equality
\[
 \mathcal{T}_{qc}(H)=\mathcal{T}_{bL}(H)
\]
and that there exists a constant $C$ such that
for every $x$ and $y$ in $\mathcal{T}_{qc}(H)$, we have
\[
d_{qc}(x,y)\leq d_{bL}(x,y)\leq C d_{qc}(x,y).
\]

 To the bi-Lipschitz Teichm\"uller space is associated a bi-Lipschitz mapping class group:
 
\begin{definition} The {\it bi-Lipschitz extended mapping class group} of the surface $X$ relatively to the base hyperbolic metric $H_0$, denoted by $\Gamma_{bL}(H_0)$, is the subgroup  of $\Gamma^*(X)$  whose elements are the equivalence classes of self-homeomorphisms of $X$ that are bi-Lipschitz with respect to the metric $H_0$. The {\it bi-Lipschitz mapping class group} of $X$, $\Gamma(X)$,  is the subgroup of order 2 of $\Gamma^*(X)$  whose elements are the equivalence classes of orientation-preserving homeomorphisms. 
\end{definition}

The group $\Gamma^*_{bL}(H_0)$ acts naturally on the Teichm\"uller space $\mathcal{T}_{bL}(H_0)$, preserving the metric   $d_{bL}$.

\begin{problem}
Determine the isometry group of the bi-Lipschitz Teichm\"uller space $\mathcal{T}_{bL}(H_0)$.
\end{problem}

This question is also open for Teichm\"uller spaces of surfaces of finite type equipped with the bi-Lipschitz metric.

 We now recall the definition of the length spectrum Teichm\"uller space.

 We consider again a
 fixed hyperbolic metric $H_0$ on the surface $X$.  
  
 Given  a homeomorphism $f:(X,H)\to (X,H')$  where $H$ and $H'$ are hyperbolic metrics, we say that $f$ is {\it length-spectrum bounded} if the following holds:
\[
K(f)= \sup_{\alpha\in\mathcal{S}} \left\{ \frac{l_{H'}(f(\alpha))}{l_{H}(\alpha)},\frac{l_{H}(\alpha)}{l_{H'}(f(\alpha))}\right\}<\infty.
\]

The quantity $K(f)$ is the {\it length-spectrum constant of $f$}. It depends only on the homotopy class of $f$.

We consider the collection of marked hyperbolic structures $(f,H)$ relative to the base structure $H_0$ and where the marking $f:H_0\to H$ is length-spectrum bounded.

There is an equivalence relation on this set, where  two marked hyperbolic structures $(f,H)$ and $(f',H')$ are equivalent if there exists an isometry or, equivalently, a length spectrum preserving homeomorphism, $f'':H\to H'$ which is homotopic to $f'\circ  f^{-1}$.  We shall say that two such structures are \emph{length spectrum equivalent}.

\begin{definition}\label{def:ls-i} The {\it length-spectrum Teichm\"uller space} $\mathcal{T}_{ls}(H_0)$ is the space of length spectrum-equivalence classes $[f,H]$ of length-spectrum bounded marked hyperbolic surfaces $(f,H)$.\end{definition}
The equivalence class $[\mathrm{Id},H_0]$ is the {\it basepoint} of the Teichm\"uller space $\mathcal{T}_{ls}(H_0)$.

The topology on the space $\mathcal{T}_{ls}(H_0)$ is the one induced by the {\it length-spectrum} metric $d_{ls}$, defined by taking the distance $d_{ls}([f,H],[f',H'])$ between two points in $\mathcal{T}_{ls}(H_0)$ represented by two marked hyperbolic surfaces $(f,H)$ and $(f',H')$ to be
        \[d_{ls}([f,H],[f',H'])=\frac{1}{2}\log K(f'\circ  f^{-1}).\]

The fact that the function $d_{ls}$ satisfies the properties of a metric is straightforward, except perhaps for the separation axiom. 

It is easy to see that there exist pairs of marked hyperbolic structures on $X$ which are not related by any length-spectrum bounded homeomorphism. 
   
There is a natural inclusion
               \[ \mathcal{T}_{bL}(H_0)\hookrightarrow \mathcal{T}_{ls}(H_0).\]
               This inclusion map continuous since for each $x$ and $y$ of $\mathcal{T}_{bL}(H_0)$,  we have
               \[                d_{ls}(x,y)\leq d_{bL}(x,y).
           \]
              This inclusion is generally strict. 
                 
                           It follows from an inequality due to Sorvali \cite{Sorvali} and Wolpert \cite{Wolpert} that there is also a natural inclusion map
\[
\mathcal{T}_{qc}(H_0)\hookrightarrow \mathcal{T}_{ls}(H_0)
\] 
which is continuous. Here also, in general, the inclusion between the two spaces is strict. 
For all these facts, for questions and results related to these inclusions, and for other results on the Teichm\"uller spaces considered, we refer the reader to the papers already mentioned,  \cite{LLPP}, \cite{ALPS1}, \cite{ALPS2}, \cite{ALPS3}, \cite{ALPS4}, \cite{Saric1} and \cite{Saric2}.

 To the length-spectrum Teichm\"uller space is associated a length-spectrum mapping class group:
 
\begin{definition} The {\it length-spectrum extended mapping class group} of the surface $X$ relatively to the base hyperbolic metric $H_0$, denoted by $\Gamma_{ls}^*(H_0)$, is the subgroup of $\Gamma(X)$  whose elements are the equivalence classes of self-homeomorphisms of $X$ that are length-spectrum bounded with respect to the basepoint metric $H_0$.
The {\it length-spectrum mapping class group} of $X$, denoted by $\Gamma_{ls}(H_0)$, is the subgroup of $\Gamma^*_{ls}(H_0)$  whose elements are the equivalence classes orientation-preserving self-homeomorphisms.\end{definition}

The group $\Gamma_{ls}^*(H_0)$ acts naturally on the Teichm\"uller space $\mathcal{T}_{ls}(H_0)$, and this action preserves the metric   $d_{ls}$.
As for the bi-Lipschitz Teichm\"uller space, the following is a natural problem:

\begin{problem}
Determine the isometry group of the length-spectrum Teichm\"uller space $\mathcal{T}_{ls}(H_0)$.
\end{problem}

This question is also open for Teichm\"uller spaces of surfaces of finite type, equipped with the length-spectrum metric.
                            
\subsection{Higher Teichm\"uller theory}

The Teichm\"uller space of a closed surface $S$ of genus $>1$ can be regarded as a subspace of the character variety 
\begin{equation}\label{R2}
\mathcal{R}_2= \mathrm{Hom}(\pi_1(S),\mathrm{PSL}(2,\mathbb{R}))/\mathrm{PSL}(2,\mathbb{R})
\end{equation}
 by considering $\mathrm{PSL}(2,\mathbb{R})$ as the isometry group of the upper half-space model of the hyperbolic plane and associating to each hyperbolic structure its holonomy homomorphism, an element of $\mathrm{Hom}(\pi_1(S),\mathrm{PSL}(2,\mathbb{R}))$. In fact, the image of the Teichm\"uller space of $S$ in $\mathcal{R}_2$ coincides with a connected component of that variety, a component which therefore consists entirely of discrete and faithful representations. Let us note that the same
 Teichm\"uller space can also be obtained in the setting of representations of $\pi_1(S)$ in the projective special unitary group $\mathrm{PSU}(1,1)$, identifying this group with the orientation-preserving isometry group of the Poincar\'e disc. The theory obtained is identical.

Higher Teichm\"uller theory is an extension of Teichm\"uller theory to representations of fundamental groups of surfaces into Lie groups other than $\mathrm{PSL}(2,\mathbb{R})$ of $\mathrm{PSU}(1,1)$.  In this setting, a \emph{higher Teichm\"uller space} is defined then as a connected component (or a union of connected components) of  a space $\mathrm{Hom}(\pi_1(S),G)/G$ that consists of equivalence classes of discrete and faithful representations, where $G$ is some simple Lie group. A question that arises naturally in this theory is therefore to find such components and to characterize them. Some questions are related to  various  structures that such components may carry (metric, symplectic, complex, combinatorial, etc.), in analogy with the structures that are known to exist on the classical Teichm\"uller spaces. The mapping class group of the surface (as the outer automorphism group of the fundamental group) acts on these components, and there are also interesting questions on the dynamics of that action. In a few cases, the action is properly discontinuous, and in a few cases there are notions of limit sets, domains of discontinuity, etc.
  
 Some of these questions are already settled, and some others are under thorough investigation, by several people. For instance, it is known that for complex (semi-)simple Lie groups, there are no connected components of the representation variety that consist entirely of equivalence classes of discrete and faithful representations.  Concerning the action, one expects that typically, when the action is not properly discontinuous, there is a decomposition of the component (like in the theory of Kleinian groups) where the action is properly discontinous on some open set and chaotic in some precise sense on its complement. This turns out to be true in some cases.

The study of the character variety involves some theory of algebraic group actions. In fact, the definition of the equivalence relation on representations that give the character variety is naturally formulated in terms of ``defining the same character", referring to the map obtained by composing the representation with the trace function. This requires, in some cases, a definition  which is more precise than the one we gave for $\mathcal{R}_2$ in (\ref{R2}).  In some cases one has to deal with the notion of Mumford quotient, which has a natural structure of an algebraic variety.

  Character varieties in the case of $G=\mathrm{PSL}(2,\mathbb{C})$ were studied in the context of Kleinian groups. In the case where $G$ is a compact group, the character variety $\mathrm{Hom}(\pi_1(S),G)/G$ was already extensively studied in the context of gauge theories.

 Although generalization in mathematics is natural, it is interesting only if it leads to new substantial results.
 
 The first substantial results in the case of non-compact Lie groups (after the cases $G=\mathrm{PSL}(2,\mathbb{R})
$ and $G=\mathrm{PSL}(2,\mathbb{C}
)$) are probably those of Hitchin (1992) who studied representations of fundamental groups of closed surfaces in 
$\mathrm{PSL}(n,\mathbb{R})$ (generalizing the classical case where $n=2$) and, more generally, in the adjoint group of a split real simple Lie group. 
   Hitchin worked in the setting of Higgs bundles. 
    These are holomorphic vector bundles, equipped with so-called \emph{Higgs fields} that appeared before in works of Hitchin and of Simpson. The terminology comes from mathematical physics, where Higgs fields are some invisible energy fields. We refer for this in the paper by Maubon in Volume II of this Handbook \cite{Maubon}.

  Hitchin  found components of the representation varieties  \[\mathcal{R}_n=\mathrm{Hom}(\pi_1(S),\mathrm{PSL}(n,\mathbb{R}))
/\mathrm{PSL}(n,\mathbb{R})
\]
   for $n\geq 2$, 
  which contain a \emph{Fuchsian locus}, that is, a subspace consisting of equivalence classes of representations obtained by composing a classical Fuchsian representations of $\pi_1(S)$ into $\mathrm{PSL}(2,\mathbb{R})
$ with the so-called ``preferred homomorphism" $\mathrm{PSL}(2,\mathbb{R})
\to \mathrm{PSL}(n,\mathbb{R})
$ defined by the irreducible representation of $\mathrm{PSL}(2,\mathbb{R})
$ into $\mathrm{PSL}(n,\mathbb{R})
$ (see \cite{Hi92}). Such a component has properties which resemble those of  Teichm\"uller space.
 and in fact, for $n=2$, Hitchin components are the Teichm\"uller components (that is, the Fuchsian components). The components in $\mathcal{R}_n$ highlighted by Hitchin are homeomorphic to $\mathbb{R}^{(2g-2)(n^{2}-1)}$, and they are called \emph{Hitchin components}. Hitchin also studied the other components. 
Soon after the work of Hitchin, Choi and Goldman showed that in the case where $G=\mathrm{PSL}(3,\mathbb{R})$, the Hitchin component consists precisely of the holonomy representations of convex real projective structures on the surface, thus answering a natural question \cite{Choi-G}. This component contains the usual Teichm\"uller space as a subspace. As a consequence, every element in that component is discrete and injective.

Following the work by Hitchin in the case of $\mathrm{PSL}(n,\mathbb{R})$, an intense activity on higher Teichm\"uller theory was conducted in the last decade by several authors, including A'Campo, Bradlow,  Burger, Fock, Gothen, Iozzi, Garc\'\i a-Prada, Goncharov, Guichard, Labourie, Loftin, Riera, Wienhard,  and there are many others. In particular, Fock and Goncharov developed a theory which is parallel to the higher Teichm\"uller theory using a notion of positiveness of representations of the fundamental group of the surface into a split semisimple algebraic group $G$ with trivial center. A similar notion of positiveness had been introduced before by G. Lusztig in his theory of canonical bases that appeared in the 1990s. Fock and Goncharov gave a more geometric version, where positiveness of a representation is defined in terms of Thurston-like  shear coordinates on the edges of surface hyperbolic ideal triangulations. They proved that these positive representations are faithful and discrete and that their moduli
space is an open cell in the space of all representations, generalizing the classical case where $G 
= \mathrm{PSL}(2, 
\mathbb{R})$. Furthermore, they developed a relation between these positive representations and cluster algebras and with the quantization theory of Teichm\"uller space. 

Labourie, in 2006, associated to every representation in a Hitchin component a curve in a projective space, which he called a \emph{hyperconvex Frenet curve},  generalizing the Veronese embedding from $\mathbb{P}(\Bbb R^2)$ into $\mathbb{P}(\mathbb{R}^n)$. Such a curve is also a generalization of the curve in $\mathbb{P}(\mathbb{R}^3)$ obtained in the case $n=3$ which is the boundary of the open set in the sphere, obtained by the developing map of a convex real projective on the surface associated to elements of the Hitchin components by works of Choi and Goldman in \cite{Choi-G}.  Labourie also introduced the notion of Anosov representations.  
 In its original form, such a representation arises as the holonomy of an Anosov structure on the underlying surface. After the work of Labourie, this notion became the basis of a new dynamical framework for the study of representations in the Hitchin components. In the case of representations in $\mathrm{PSL}(n,\mathbb{R})$,  Labourie showed that the representations are discrete and injective and that every element in the image is diagonalisable with eigenvalues being all real and with distinct absolute values.  An Anosov representation is an analogue of a convex cocompact representation in the setting of higher rank Lie groups. Labourie showed that a Hitchin component in the case of  
$\mathrm{PSL}(n,\mathbb{R})$ consists entirely of Anosov representations (and therefore they are discrete and faithful) and that a representation in such a component is a quasi-isometric embedding. He also showed that Anosov representations form a set on which the mapping class group acts property discontinuously.   In particular,  the mapping class group acts properly discontinuously on the Hitchin component.   
 
 In 2008, Guichard and Wienhard completed the characterization of the representations in the Hitchin components in the case where $G=\mathrm{PSL}(4,\mathbb{R})$ as properly convex foliated projective 
structures on the unit tangent bundle of the surface \cite{G2008}. Labourie, Guichard and Wienhard showed that in the case where $G$ is a rank one semisimple Lie group, a representation is Anosov if and only if it is a quasi-isometric embedding, and moreover that this representation is Anosov if and only if its image is a convex cocompact group. Guichard and Wienhard also developed the theory of Anosov representations in the context of representation theory of hyperbolic groups.

Guichard and Wienhard obtained in \cite{GW2010} a description of the topological components of the character variety for $G=\mathrm{PSL}(4,\mathbb{R})$. The number of such topological components had been counted before by Garc\'\i a Prada and Riera in \cite{GPR}.

 There is also an extensive literature on the symplectic structure of the character variety, first introduced by Atiyah and Bott for the case of compact groups \cite{AB1983} and then by Goldman \cite{Goldman84} for non-compact groups, who also made the relation with the Weil--Petersson K\"ahler form. This led in particuar to the Witten formulae on the symplectic volume of these character varieties of compact groups, with relations with 3-manifolds and the Chern-Simons invariant.

As we already mentioned, the mapping class group acts on the various components of the representation variety, and  of course  one should expect rigidity theorems for these actions, when the representation variety is equipped with extra structure. I will discuss one instance of such a structure, which is the recently defined \emph{pressure metric}. 
  
  The definition starts with an alternative definition of the Weil-Petersson metric on Teichm\"uller space which was given in 1985 by Thurston. 
  
  Thurston defined a Riemannian metric on Teichm\"uller space where the scalar product of two vectors at a given point represented by a hyperbolic surface $S$ is the second derivative with respect to the earthquake flow along these vectors of the length of a ``uniformly distributed sequence" of closed geodesics in $S$. This is motivated by the fact that the geodesic length function on Teichm\"uller space associated to a simple closed curve is convex along earthquake paths (this is a result of Kerckhoff \cite{Kerckhoff1983}), and therefore such a function  behaves like (the square of) a distance function measured from the minimum, with the second derivative at that minimum playing the role of a metric tensor. Taking the length of a uniformly distributed sequence of closed geodesics removes the dependence on  the chosen length function. Soon after that, Wolpert showed in \cite{Wolpert1986} that this metric on Teichm\"uller space  is nothing else but the Weil-Petersson metric. More recently, Bridgeman and Taylor defined in \cite{BT2008} a semi-definite metric on quasi-Fuchsian space (a complexification of Teichm\"uller space defined via Bers' simultaneous uniformization theorem) to which the space of geodesic currents and of length functions can be extended. They showed that this semi-definite metric is an extension of the Weil-Petersson metric on Teichm\"uller space. The construction of Bridgeman and Taylor involves taking the second derivative of the Hausdorff dimension of the limit set. This  is again based on the fact that the Hausdorff dimension function reaches its minimum on Teichm\"uller space and that its Hessian defines a semi-definite metric. Bridgeman and Taylor used this metric to show that the second derivative of the Hausdorff dimension $h$ of the limit set in some directions $b$ (the so-called pure bending directions) is at least equal to the Weil-Petersson metric on the associated twist vector $t$. In other words $h''(b) \geq g(t,t)$ for $b = J.t$ where $t$ is a vector tangent to the Fuchsian locus and $J$ is the complex structure on the quasi-Fuchsian locus.  
  
 Building on this work of Bridgeman and Taylor and using the thermodynamic formalism developed by Bowen in the 1970s \cite{Bowen1973} \cite{Bowen1975},  McMullen gave in \cite{McMullen22008} a thermodynamic interpretation of that metric, which is therefore called a ``pressure metric".
McMullen also showed that the inequality $h''(b) \geq g(t,t)$ is in fact an equality and that for the Markov coding in the quasi-Fuchsian case, this pressure metric is equal to the metric $G$ defined by Bridgeman and Taylor. In a subsequent paper, \cite{Bridgeman2010}, Bridgeman used McMullen's  approach to study $G$ and the critical points of the Hausdorff dimension function. The quasi-Fuchsian extension $G$ is not a Riemannian metric but is a path metric in which some vectors have zero length. Bridgeman showed that the singular locus consists exactly of the pure bending vectors on the Fuchsian locus.

 McMullen also interpreted this metric as (a multiple of) the Hessian of the dimension of the push-forward of Lebesgue measure on the circle. At the same time, he obtained a Weil-Petersson-type metric on the space of H\"older continuous functions modulo coboundaries. The normalized Weil-Petersson metric on Teichm\"uller space is then obtained  as a pullback of the pressure metric via the ``thermodynamic" embedding of Teichm\"uller space into that space.  
 
 More recently, the pressure metric formulation of the Weil-Petersson metric was generalized by Bridgeman, Canary, Labourie and  Sambarino \cite{Pressure} to a mapping class group invariant metric on the Hitchin components of the character variety of representations of the fundamental group of the surface into $\mathrm{SL}(n,\mathbb{R})$ which restricts to a multiple the Weil-Petersson metric on the Fuchsian locus. 
  
 In fact, the authors in \cite{Pressure} define more generally an $\mathrm{Out}(\Gamma)$-invariant Riemannian metric on the set of smooth points of the deformation spaces of convex, irreducible representations of a word hyperbolic group $\Gamma$ into $\mathrm{SL}(m,\mathbb{R})$. They also use the so-called Pl\"ucker representation to produce metrics on deformation spaces of convex cocompact representations of the hyperbolic group into  $\mathrm{PSL}(2,\mathbb{C})$ and on the smooth points of deformation spaces of Zariski dense Anosov representations into arbitrary semi-simple Lie groups. They generalize in several directions the work that was done in the Fuchsian and quasi-Fuchsian cases by Bridgeman and by McMullen.
  
There is also a Riemannian metric on a space which contains Teichm\"uller space, which was defined by Darwishzadeh and Goldman in \cite{DG1996} on the space $\mathcal{B}(S)$ of equivalence classes of real projective structures on a closed surface $S$. The definition uses Hodge theory and a classical construction of a canonical metric due to Koszul and Vinberg. By a result of Goldman \cite{Goldman1990}, the  associated space $\mathcal{B}(S)$ is homeomorphic to an open cell of dimension $16(g-1)$, where $g$ is the genus of $S$. Labourie and Loftin showed independently that an element of this space is parametrized by a pair (Riemann surface, holomorphic cubic differential), see \cite{Labourie2007} and \cite{Loftin2001}, see also the survey \cite{Loftin-H}. Teichm\"uller space is embedded in $\mathcal{B}(S)$ as the space of pairs for which the cubic differential is identically zero.
 Qiongling Li showed recently that this Riemannian metric restricts to the Weil-Petersson metric on Teichm\"uller space. She also defined a new metric on $\mathcal{B}(S)$, which she called the \emph{Loftin metric},  and she showed that Teichm\"uller space  equipped with the Weil-Petersson metric, sits as a totally geodesic subspace in $\mathcal{B}(S)$ equipped with the Loftin metric.  
    
 There are several questions which arise in this theory, and the most natural one, in relation with the subject of this survey, is to study the automorphisms of these metrics defined on subspaces of the representation spaces, with respect to various structures on these subspaces, and to find or at least to conjecture what are the rigidity results that hold for them.

Let us note finally that there are several generalizations of actions on the character variety of representations of the free group and to other finitely generated groups (we already mentioned hyperbolic groups) into Lie groups, see \cite{Goldman2006} and the surveys by Wienhard \cite{W2004},  by Wentworth \cite{Went-Survey} and by Burger-Iozzi-Wienhard \cite{BIW}. We refer to the survey by Canary \cite{Canary-K} in this volume and to the references there, concerning the representation theory of outer automorphisms of word hyperbolic groups into semi-simple Lie groups.

     \subsection{The Grothendieck-Teichm\"uller theory}\label{s:Gro}

  Grothendieck's theory, also called the Grothendieck-Teichm\"uller theory,  makes relations between, on the one and, the theory of Teichm\"uller spaces and mapping class groups, and, on the other hand, the study of the \emph{absolute Galois group} $\Gamma_{\mathbb{Q}}$ of the field  $\mathbb{Q}$ of rational numbers, that is, the group of automorphisms of the algebraic closure $\overline{\mathbb{Q}}$ of $\mathbb{Q}$. In this section, we mention a few elements of this theory, because there is much work that can be done in this setting. 
 
 Let us first recall that understanding the field extensions of $\mathbb{Q}$, and in particular the finite ones (called \emph{number fields}), is one of the major questions in number theory. For instance, the so-called \emph{inverse Galois problem}, which dates from the nineteenth century, asks whether every finite group is a \emph{Galois group}, that is, whether it is the automorphism group of some number field. A famous result of Shafarevich gives a positive response to this question in the case of finite solvable groups.  Originally, the importance of the question of the realization of finite groups as automorphisms of number fields stems from its relation with the solvability of equations by radicals. 
 The group $\Gamma_{\mathbb{Q}}$ contains (in principle) information on all number fields and it  is considered as the most important group in number theory, but almost nothing is known about its structure.  

 To understand the absolute Galois group, one naturally have to look for geometric actions of that group. 
An example of a geometric action of $\Gamma_{\mathbb Q}$ is the one on the algebraic fundamental group of an algebraic variety $X$ defined over $\mathbb Q$. We recall that the algebraic fundamental group of $X$ is the profinite completion of its topological fundamental group. We now review a few facts on profinite groups and group completions because they appear in the Grothendieck-Teichm\"uller theory. In fact, the absolute Galois group is itself a profinite group.

A profinite group is a Hausdorff compact totally disconnected topological group. Equivalently, this is a topological group which is an inverse limit of a system of finite groups. Thus, in some sense, profinite groups are asymptotic groups with respect to finite groups. The group of $p$-adic integers $\mathbb{Z}_p$ equipped with addition is an example of a profinite group; it is the inverse limit of the system of finite groups $(\mathbb{Z}/p^n\mathbb{Z})$ equipped with the natural projection maps $\mathbb{Z}/p^n\mathbb{Z}\to \mathbb{Z}/p^m\mathbb{Z}$ for $n\geq m$.

To an arbitrary group $G$ is associated its \emph{profinite completion}, denoted by $\widehat{G}$, defined as the inverse limit of the system of groups $G/N$ where $N$ runs over the normal subgroups of finite index of $G$. 
 The partial order defined by inclusion between subgroups induces a system of natural homomorphisms between the corresponding quotients. This makes the set of groups $G/N$ an inverse system. 
There is a natural homomorphism $G\to \widehat{G}$ whose image is dense in $\widehat{G}$. This homomorphism is injective if and only if $G$ is residually finite. 

Normal subgroups of finite index are in one-to-one correspondence with finite covering spaces and this gives a geometric representation of the operation of taking the profinite completion. This is also a clue to why these groups appear in the Teichm\"uller theory of Grothendieck.
 
 Now we come back to the work of Grothendieck.
 
  Grothendieck became interested in the topology of surfaces at about the same time Thurston developed his theory,  and he expressed  this interest in late writings which he  did not develop because he had already put an end to his activity as a mathematician. His ideas on the absolute Galois group starts from the conviction that this group should be studied through its action on profinite versions of moduli spaces and of mapping class groups rather than on number fields directly. He  introduced in this theory the object he called the \emph{Teichm\"uller tower} (which is called now the \emph{Grothendieck-Teichm\"uller tower}). This is a collection of mapping class groups of all surfaces of variable finite topological type, together with the natural homomorphisms between them that arise from maps between moduli spaces of corresponding surfaces, which are themselves induced by natural maps between the surfaces. For instance, the maps may be induced from inclusions between surfaces, or by closing up some punctures. There is an (outer) automorphism group of such a tower, which is given by a collection of automorphisms of the various mapping class groups that constitute it, and which are compatible (up to inner automorphisms) with the homomorphisms of the tower; see the theory presented in \cite{HLS}. See also the survey \cite{AJP2}. We remind the reader that towers, e.g. towers of field extensions, often occur in Galois theory, and in particular in class field theory. Several important questions in algebraic geometry and number theory involve the construction of towers. In view of this, it is not surprising that Grothendieck made this definition. The main conjecture of Grothendieck in this setting is that the automorphism group of the Teichm\"uller tower is isomorphic to the absolute Galois group. 
  
We quote the following from Grothendieck's \emph{Esquisse d'un programme} (English translation by L. Schneps):

\begin{quote}\small
It is more the system of \emph{all} the multiplicities $M_{g,\nu}$ for variable $g,\nu$, linked together by a certain number of operations (such as the operations of ``plugging holes", i.e. ``erasing" marked points, and of ``glueing", and the inverse operations), which are the reflection in absolute algebraic geometry in characteristic zero (for the moment) of geometric operations familiar from the point of view of topological or conformal ``surgery" of surfaces. Doubtless the principal reason is that this very rich geometric structure on the system of ``open" modular multiplicities $M_{g,\nu}$ is reflected in an analogous structure on the corresponding fundamental groupoids, the ``Teichm\"uller groupoids" $\widehat{T}_{g,\nu}$, and that these operations on the level of $\widehat{T}_{g,\nu}$ are sufficiently intrinsic for the Galois group $\Gamma$ on $\overline{\mathbb{Q}}/\mathbb{Q}$ to act on the whole ``tower" of Teichm\"uller groupoids, respecting all these structures. Even more extraordinary, this action is \emph{faithful} -- indeed, it is already faithful on the first non-trivial ``level" of this tower, namely, $\widehat{T}_{0,4}$ -- which also means, essentially, that the outer action of $\Gamma$ on the fundamental group $\widehat{\pi}_{0,3}$ of the standard projective line $\mathbb{P}^1$ over $\mathbb{Q}$ with the three points 0, 1 and $\infty$ removed, is already faithful. Thus, \emph{the Galois group $\Gamma$ can be realized as an automorphism group of a very concrete profinite group}, and moreover respects certain essential structures of this group. It follows that an element of $\Gamma$ can be ``parametrized" (in various equivalent ways) by a suitable element of this profinite group  $\widehat{\pi}_{0,3}$  (a free profinite group on two generators), or by a system of such elements, these elements being subject to certain simple necessary (but doubtless not sufficient) conditions for this or these elements to really correspond to an element of $\Gamma$. One of the most fascinating tasks here is precisely to discover necessary \emph{and} sufficient conditions on an exterior automorphism of $\widehat{\pi}_{0,3}$, i.e. on the corresponding parameter(s), for it to come from an element of $\Gamma$ -- which would give a ``purely algebraic" description, in terms of profinite groups and with no reference to the Galois theory of number fields, to the Galois group $\Gamma=\overline{\mathbb{Q}}/\mathbb{Q}$.
\end{quote}
    
             We refer the interested reader to the recent surveys \cite{Lochak2012}, \cite{Schneps} and \cite{AJP2} and to the references there for some work done in this direction.

The relation of Grothendieck's ideas with the Thurston type topology of surfaces was never stressed enough.  There is much to do in this field. 
 Instead of towers of Teichm\"uller spaces, one may consider towers of other spaces: combinatorial, metric, etc. (in fact, one can consider \emph{all} the spaces we mentioned on which mappings class groups of surfaces act) on which the mapping class groups of surfaces finite type act and study the automorphism groups of these towers.

   \subsection{The solenoid theory}

   The ideas behind this theory are closely related to Grothendieck's ideas which we described in \S \ref{s:Gro}, though none of the papers that study the solenoid which we quote below mention Grothendieck's work.

The \emph{solenoid}  is a  fiber space over a closed Riemann surface of genus $\geq 2$ obtained through the construction of a tower of coverings. It was introduced by Sullivan in the early 1990s (see \cite{Sullivan1993}), as  the inverse limit of a tower of finite-sheeted pointed coverings of a fixed closed surface of genus $\geq 2$. 
More precisely, we start with a closed surface $S$ of genus $\geq 2$ equipped with a basepoint $x$ and we consider the set of all  basepoint-preserving unbranched coverings of $S$, $\{(S',x')\to (S,x)\}$. There is a natural partial order relation on this collection of coverings: given two coverings  $\pi_i:(S_i,x_i)\to (S,x)$ and  $\pi_j:(S_j,x_j)\to (S,x)$, we write $\pi_i\leq \pi_j$ if $\pi_j$ factors through a covering $\pi_{j,i}: (S_j,x_j)\to (S_i,x_i)$, that is, if $\pi_j=\pi_i\circ \pi_{j,i}$. We obtain an inverse directed set, and the solenoid is its topological  inverse limit.
The homeomorphism type of this inverse limit does not depend on the choice of the base surface (provided it has negative Euler characteristic)  and the solenoid can be thought of as the ``universal closed surface" because it fibers over any surface that is in the tower of coverings of $S$ (and not only over the base surface $S$). 

The solenoid locally is homeomorphic to the product of a surface with a Cantor set. This is a two-dimensional topological analogue of a geodesic lamination with non-closed leaves on a hyperbolic surface. For this reason, the solenoid is also called a \emph{surface lamination}. The path-connected components of this lamination are called \emph{leaves}, and they are all homeomorphic to the open disc. Each leaf is dense in the solenoid. The leaf which contains the basepoint is called the \emph{baseleaf}.  The solenoid is also a principal $\widehat{\pi_1(S)}$-bundle over the surface $S$, where $\widehat{\pi_1(S)}$ is the profinite completion of the fundamental group $\pi_1(S)$.  

There are natural notions of differentiable, complex, hyperbolic, etc. structures on the solenoid. Such a structure is defined in terms of an \emph{atlas} of the soleniod considered as a surface lamination. More precisely, an atlas of the solenoid is a collection of pairs $\{(V_i, \psi_i)\}_{\i \in \mathcal{I}}$, where each $V_i$ is an open set of the solenoid and $\psi_i:V_i\to U_i\times T_i$ a homemorphism, with $U_i$ being a 2-dimensional disc and $T_i$ a Cantor set, and where the transition maps 
$\psi_j\circ \psi_i^{-1}$, defined on the appropriate domains, send discs into discs.
A differentiable, complex, hyperbolic, etc. structure on the solenoid is then defined accordingly, with transition maps differentiable, complex, hyperbolic, etc. in the disc direction, and continuous in the Cantor set direction. All this is due to Sullivan \cite{Sullivan1993}.

 The work of Candel on the uniformization of surface laminations \cite{Candel1993} shows that each complex structure on the solenoid contains a unique hyperbolic structure.

A complex or a hyperbolic structure on the solenoid may or may not be a lift of a complex or a hyperbolic structure on a closed Riemann surface under the fiber map. However, these special structure play a special role and they are said to be \emph{transversely locally constant}. 

There is a notion of quasiconformal map from the compact solenoid to itself equipped with two complex structures, and there is a corresponding Teichm\"uller space, defined as a space of equivalence classes of marked complex structures. This space can also be obtained as the closure of the stack of all Teichm\"uller spaces of surfaces of genus $>1$. It is also a universal object in  the sense that one can canonically embed in it every classical Teichm\"uller space of a hyperbolic surface.  The Teichm\"uller space of the solenoid admits a complete  metric which is defined in a way similar to the Teichm\"uller metric of a surface. 
 This space is infinite-dimensional but, unlike the classical universal Teichm\"uller space, it is separable, a fact which is remarkable because all the classical Teichm\"uller spaces are either finite-dimensional or infinite-dimensional and non-separable. We refer the reader to the survey by \v{S}ari\'c \cite{Saric}.

There is a notion of \emph{mapping class group} of the solenoid, defined as the group of isotopy classes of quasiconformal self-maps of this space.  This  group acts on the Teichm\"uller space of the solenoid, preserving the Teichm\"uller metric.  There is a natural group of \emph{baseleaf preserving mapping class group}.

In \cite{MS2005}, Markovic and \v{S}ari\'c study quasiconformal maps between complex solenoids. They show  that  if two such maps are homotopic, then they are isotopic by a uniformly quasiconformal isotopy. This is a property which is analogous to a property shared by maps between compact and non-compact Riemann surfaces. The case of surfaces of infinite type  was also treated by Earle and McMullen in \cite{EM1988}. Markovic and \v{S}ari\'c also show that any homeomorphism of the solenoid is homotopic to a quasiconformal self-map. An analogue of the last property holds for maps between closed surfaces but it does not hold in general for open surfaces. The paper \cite{MS2005} contains several results on limits of quasiconformal maps as well as an analogue of the Nielsen realization theorem for the mapping class group of the solenoid. Markovic and \v{S}ari\'c also studied the dynamics of the action of the solenoid on its Teichm\"uller space. They proved that there exists a dense subset of that space such that the orbit of the baseleaf preserving mapping class group of any point of this subset has an accumulation point. In particular, this action is not properly discontinuous.

The following rigidity result for the solenoid mapping class group was obtained by Odden:
\begin{theorem}[\cite{Odden1997} and \cite{Odden2004}] \label{th:virt}
The baseleaf preserving mapping class group of the  solenoid is naturally isomorphic to the virtual automorphism group of the fundamental group $\pi_1(S)$ of the base surface.
\end{theorem}

The \emph{virtual automorphism group} of a group can be roughly described as a group of isomorphisms between finite index subgroups of $\pi_1(S)$ up to a certain equivalence relation. More precisely, a \emph{virtual automorphism} of a group is an equivalence class of isomorphisms between two (possibly the same) subgroups of finite index, where two such isomorphisms are said to be equivalent if they agree on some finite-index subgroup of their intersection. The virtual automorphism group  is also called the {\it abstract commensurator group} of the given group. 

The result of Theorem \ref{th:virt} can be considered as a genus-independent analogue of the Dehn-Nielsen Theorem saying that the mapping class group of a closed surface is canonically isomorphic to the outer automorphism of the fundamental group.

 There is a correspondence between equivalence classes of virtual automorphisms of the fundamental group of a surface and homotopy classes of basepoint-preserving homeomorphisms between pointed coverings of the surface. This gives a strong relation between the solenoid and the virtual automorphism group of $\pi_1(S)$.

 The mapping class group of the solenoid acts on the Teichm\"uller space of that space. This action has been studied by Biswas, Nag and Sullivan in \cite{BNS1996} and after that by others.

\v{S}ari\'c developed the theory of the punctured solenoid, see \cite{Saric} and \cite{PS2008}. The difference with the unpunctured version is that now branching is permitted now over the punctures (or distinguished points). The Teichm\"uller space of the punctured solenoid is also a separable Banach space admitting a complete Teichm\"uller metric, and it is also a universal object in  the sense that one can canonically embed in it every classical Teichm\"uller space of a hyperbolic punctured surface. Penner's decorated theory and his  $\lambda$-length coordinates can be adapted to an appropriate decorated Teichm\"uller space of  the punctured solenoid, and the structure of the decorated Teichm\" uller space for finite type surfaces survives in this setting. We refer the reader to the survey \cite{Penner-Hand} by Penner.  We also refer to the paper by Bonnot, Penner and \v{S}ari\'c \cite{BPS} in which they describe a cellular action of the baseleaf preserving mapping class group of the punctured solenoid which leads to a presentation of that group.

 Sullivan introduced the solenoid for understanding universal properties of dynamical systems, and also as a tool for attacking the Ehrenpreis conjecture. We recall that this conjecture (which was proved recently by Kahn and Markovich) says that for any two closed Riemann surfaces $S_1$ and $S_2$, there are finite unbranched covers of $S_1$ and $S_2$ which are of the same type and whose Teichm\"uller distance is arbitrarily small.
 
  The theory of the solenoid, even though it was not used in the proof of the Ehrenpreis conjecture, found several applications. For instance, in the paper \cite{BNS1996}, Biswas, Nag and Sullivan used this object to obtain a coherent, genus-independent description of the so-called Mumford isomorphisms, which hold in each genus, between certain higher DET bundles over the moduli space of a closed Riemann surface of genus $g$ and certain fixed (genus-independent) tensor powers of the Hodge line bundle over this moduli space (the Hodge  bundle being the first of these DET bundles). There are also relations with theoretical physics, namely, with Polyakov's  bosonic string theory.

Part of the Teichm\"uller theory of the solenoid has been so far developed, but there are other interesting research problems related to this object. One of them would be the following:
\begin{problem}
Discover the other facets of the Teichm\"uller theory of the solenoid, in  particular the hyperbolic point of view (Thurston's asymmetric metric, the length spectrum metric, etc.). 
\end{problem}
Another problem is the following:
\begin{problem}
Define (combinatorial, topological, etc.) spaces on which the mapping class group of the solenoid acts, in analogy with those that we surveyed in the previous sections, study these actions and the rigidity of the automorphism groups of these spaces.
\end{problem}

Finally, let us mention that there are number-theoretic questions related to the solenoid, in particular in connection with the study of lattices in $\mathrm{SL}(2,\mathbb{Z})$ (see e.g. \cite{Penner-Hand}). 
These questions and the profinite constructions in the theory of the solenoid lead to the following:

\begin{question}
 How does the solenoid theory fit exactly into the profinite theory of Grothendieck, and in particular, what is the relation between the solenoid mapping class group and the absolute Galois group ?
 \end{question}

\

\end{document}